\pdfoutput=1		

\documentclass{article}  
\usepackage{commath}
\usepackage{graphicx}
\usepackage{enumitem}
\usepackage{tikz-cd}
\usepackage{subcaption}

\usepackage{amsmath}
\usepackage{amssymb}
\usepackage{amsthm}
\usepackage{bbm}
\usepackage{mathrsfs}

\let\emptyset\varnothing

\DeclareMathOperator{\Aut}{Aut}

\DeclareMathOperator{\Crit}{Crit}
\DeclareMathOperator{\CZ}{CZ}

\DeclareMathOperator{\Diff}{Diff}

\DeclareMathOperator{\ev}{ev}
\DeclareMathOperator{\Fix}{Fix}

\DeclareMathOperator{\grad}{grad}

\DeclareMathOperator{\Ham}{Ham}
\DeclareMathOperator{\Hess}{Hess}
\DeclareMathOperator{\HF}{HF}
\DeclareMathOperator{\HM}{HM}

\DeclareMathOperator{\id}{id}
\DeclareMathOperator{\im}{im}
\DeclareMathOperator{\ind}{ind}
\DeclareMathOperator{\Int}{Int}

\DeclareMathOperator{\loc}{loc}

\DeclareMathOperator{\ord}{ord}
\DeclareMathOperator{\pr}{pr}

\DeclareMathOperator{\RFC}{RFC}
\DeclareMathOperator{\RFH}{RFH}

\DeclareMathOperator{\SH}{SH}

\DeclareMathOperator{\Sp}{Sp}

\DeclareMathOperator{\Spec}{Spec}
\DeclareMathOperator{\supp}{supp}
\DeclareMathOperator{\Symp}{Symp}

\newcommand{\bld}[1]{\boldmath\textit{\textbf{#1}}\unboldmath}

\newtheoremstyle{main} 		             	 		
  	{}	                                     		
  	{}	                                    		
  	{\itshape}			                     		
  	{}        	                             		
  	{\boldmath\bfseries}   	                         		
  	{.}            	                        		
  	{ }           	                         		
  	{\thmname{#1}\thmnumber{ #2}\thmnote{ (#3)}}	
\theoremstyle{main}
\newtheorem{definition}{Definition}[section]
\newtheorem{proposition}{Proposition}[section]
\newtheorem{corollary}{Corollary}[section]
\newtheorem{theorem}{Theorem}[section]
\newtheorem{lemma}{Lemma}[section]
\newtheoremstyle{nonit} 		             	 		
  	{}	                                     		
  	{}	                                    		
  	{}			                     		
  	{}        	                             		
	{\boldmath\bfseries}   	                         		
  	{.}            	                        		
  	{ }           	                         		
  	{\thmname{#1}\thmnumber{ #2}\thmnote{ (#3)}}	
\theoremstyle{nonit}
\newtheorem{remark}{Remark}[section]

\newtheoremstyle{ex} 		             	 		
  	{}	                                     		
  	{}	                                    		
  	{\small}			                     		
  	{}        	                             		
  	{\bfseries\boldmath}   	                         		
  	{.}            	                        		
  	{ }           	                         		
  	{\thmname{#1}\thmnumber{ #2}\thmnote{ (#3)}}	
\theoremstyle{ex}

\usepackage[backend=bibtex, style=alphabetic]{biblatex}
\bibliography{bibliography}
\usepackage[bookmarksopen=true,bookmarksnumbered=true]{hyperref}
\hypersetup{
  colorlinks   = true, 
  urlcolor     = blue, 
  linkcolor    = blue, 
  citecolor    = blue 
}

\begin{document}

\title{First Steps in Twisted Rabinowitz--Floer Homology
}


\author{Yannis B\"ahni
}

\maketitle

\begin{abstract}
Rabinowitz--Floer homology is the Morse--Bott homology in the sense of Floer associated with the Rabinowitz action functional introduced by Kai Cieliebak and Urs Frauenfelder in 2009. In our work, we consider a generalisation of this theory to a Rabinowitz--Floer homology of a Liouville automorphism. As an application, we show the existence of noncontractible periodic Reeb orbits on quotients of symmetric star-shaped hypersurfaces. In particular, our theory applies to lens spaces.
\end{abstract}

\section{Introduction}
\label{sec:introduction}

In this paper, we introduce an analogue of the twisted Floer homology \cite{uljarevic:liouville:2017} in the Rabinowitz--Floer setting. See the excellent survey article \cite{albersfrauenfelder:rfh:2012} for a brief introduction to Rabinowitz--Floer homology and \cite{schlenk:floer:2019} for an overview of common different Floer theories. Following  \cite{cieliebakfrauenfelder:rfh:2009} and \cite{albersfrauenfelder:rfh:2010}, we construct a Morse--Bott homology for a suitable twisted version of the standard Rabinowitz action functional, generalising the standard Rabinowitz action functional.

\begin{theorem}[Twisted Rabinowitz--Floer Homology]
	Let $(M,\lambda)$ be the completion of a Liouville domain $(W,\lambda)$ and let $\varphi \in \Diff(W)$ be of finite order near the boundary $\partial W$ with $\varphi(\partial W) = \partial W$ and $\varphi^* \lambda - \lambda = df_\varphi$ for some smooth compactly supported function $f_\varphi \in C^\infty_c(\Int W)$ in the interior of $W$.

	\begin{enumerate}[label=\textup{(\alph*)}]
		\item The semi-infinite dimensional Morse--Bott homology $\RFH^\varphi(\partial W,M)$ in the sense of Floer of the twisted Rabinowitz action functional exists and is well-defined. Moreover, twisted Rabinowitz--Floer homology is invariant under twisted homotopies of Liouville domains.
		\item If $\partial W$ is simply connected and does not admit any nonconstant twisted Reeb orbit, then $\RFH^\varphi_*(\partial W,M) \cong \operatorname{H}_*(\Fix(\varphi\vert_{\partial W});\mathbb{Z}_2)$.
		\item If $\partial W$ is displaceable by a compactly supported Hamiltonian symplectomorphism in $(M,\lambda)$, then $\RFH^\varphi(\partial W,M) \cong 0$.
	\end{enumerate}
	\label{thm:twisted_rfh}
\end{theorem}

Part (a) will be proven in Sections \ref{sec:definition} and \ref{sec:invariance}, in particular Theorem \ref{thm:invariance}, part (b) is the content of Proposition \ref{prop:fixed_points} and finally part (c) is the content of Theorem \ref{thm:displaceable}. On the one hand, twisted Rabinowitz--Floer homology does generalise standard Rabinowitz--Floer homology as
\begin{equation*}
	\RFH^{\id_W}(\partial W,M) = \RFH(\partial W,M),
\end{equation*}
\noindent on the other hand, twisted Rabinowitz--Floer homology can be used to prove existence of noncontractible periodic Reeb orbits. Related results appeared in \cite[Corollary~1.6~(iv)]{sandon:reeb:2020} and \cite[Theorem~1.2]{liuzhang:noncontractible:2021}.

\begin{theorem}
	Let $\Sigma \subseteq \mathbb{C}^n$, $n \geq 2$, be a compact and connected star-shaped hypersurface invariant under the rotation	
	\begin{equation*}
		\varphi \colon \mathbb{C}^n \to \mathbb{C}^n, \quad \varphi(z^1,\dots,z^n) := \del[1]{e^{2\pi i k_1/m}z^1,\dots,e^{2\pi i k_n/m}z^n}
	\end{equation*}
	\noindent for some even $m \geq 2$ and $k_1,\dots,k_n \in \mathbb{Z}$ coprime to $m$. Then $\Sigma/\mathbb{Z}_m$ admits a noncontractible periodic Reeb orbit.
	\label{thm:noncontractible}
\end{theorem}

The proof is straightforward, once we have computed the $\mathbb{Z}_m$-equivariant twisted Rabinowitz--Floer homology of the spheres $\mathbb{S}^{2n - 1} \subseteq \mathbb{C}^n$. Indeed, by invariance we may assume that $\Sigma = \mathbb{S}^{2n - 1}$, as $\Sigma$ is star-shaped. Then we use the following elementary topological fact (see Lemma \ref{lem:noncontractible} below). Let $\Sigma$ be a simply connected topological manifold and let $\varphi \colon \Sigma \to \Sigma$ be a homeomorphism of finite order $m$ that is not equal to the identity. If the induced discrete action
\begin{equation*}
	\mathbb{Z}_m \times \Sigma \to \Sigma, \qquad [k] \cdot x := \varphi^k(x)
\end{equation*}
\noindent is free, then $\pi \colon \Sigma \to \Sigma/\mathbb{Z}_m$ is a normal covering map \cite[Theorem~12.26]{lee:tm:2011}. For $x \in \Sigma$ define the \bld{based twisted loop space of $\Sigma$ and $\varphi$} by
\begin{equation*}
	\mathscr{L}_\varphi(\Sigma,x) := \cbr[0]{\gamma \in C(I,\Sigma) : \gamma(0) = x \text{ and } \gamma(1) = \varphi(x)},
\end{equation*}
\noindent where $I := \intcc[0]{0,1}$. Then we have the following result. See Figure \ref{fig:twisted_lift}.

\begin{lemma}
	If $\gamma \in \mathscr{L}_\varphi(\Sigma,x)$ for some $x \in \Sigma$, then $\pi \circ \gamma \in \mathscr{L}(\Sigma/\mathbb{Z}_m,\pi(x))$ is not contractible. Conversely, if $\gamma \in \mathscr{L}(\Sigma/\mathbb{Z}_m,\pi(x))$ is not contractible, then there exists $1 \leq k < m$ such that $\tilde{\gamma}_x \in \mathscr{L}_{\varphi^k}(\Sigma,x)$ for the unique lift $\tilde{\gamma}_x$ of $\gamma$ satisfying $\tilde{\gamma}_x(0) = x$.
	\label{lem:noncontractible}
\end{lemma}

\begin{figure}[h!tb]
	\centering
	\includegraphics[width=.69\textwidth]{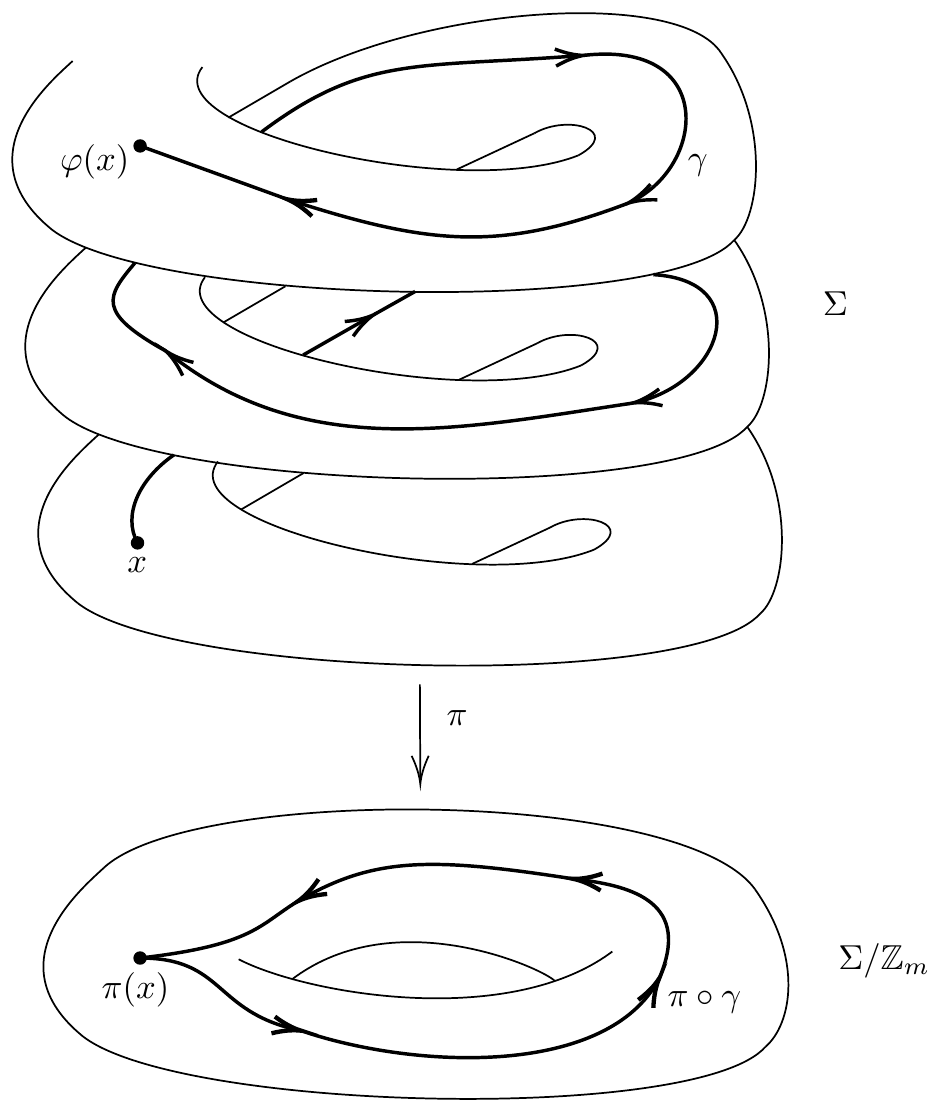}
	\caption{The projection $\pi \circ \gamma \in \mathscr{L}(\Sigma/\mathbb{Z}_m,\pi(x))$ of $\gamma \in \mathscr{L}_\varphi(\Sigma,x)$ is not contractible for the deck transformation $\varphi \neq \id_\Sigma$.}
	\label{fig:twisted_lift}
\end{figure}

For a more detailed study of twisted loop spaces of universal covering manifolds as well as a proof of Lemma \ref{lem:noncontractible} see Appendix \ref{twisted_loops_on_universal_covering_manifolds}. We put Theorem \ref{thm:noncontractible} into context. If $\Sigma^{2n - 1}/\mathbb{Z}_m$ satisfies the index condition
\begin{equation}
	\label{eq:index_condition}
	\mu_{\CZ}(\gamma) > 4 - n
\end{equation}
\noindent for all contractible Reeb orbits $\gamma$, the $\bigvee$-shaped symplectic homology $\check{\SH}(\Sigma)$ can be defined in the positive cylindrical end $\intco{0,+\infty} \times \Sigma$ by \cite[Corollary~3.7]{uebele:reeb:2019}. If $\Sigma/\mathbb{Z}_m$ admits a Liouville filling $W$, then we have 
\begin{equation*}
	\check{\SH}_*(\Sigma/\mathbb{Z}_m,M) \cong \RFH_*(\Sigma/\mathbb{Z}_m,M),
\end{equation*}
\noindent where $M$ denotes the completion of $W$. Note that even in the case of lens spaces this need not be the case, as for example $\mathbb{RP}^{2n - 1}$ is not Liouville fillable for any odd $n \geq 2$ by \cite[Theorem~1.1]{ghiggininiederkrueger:fillings:2020}. As the index condition \eqref{eq:index_condition} is only required for contractible Reeb orbits and they come from the universal covering manifold $\Sigma$, we can say something in the case where $\Sigma$ is strictly convex. Indeed, the Hofer--Wysocki--Zehnder Theorem \cite[Theorem~12.2.1]{frauenfelderkoert:3bp:2018} then implies that $\Sigma$ is dynamically convex, that is,
\begin{equation*}
	\mu_{\CZ}(\gamma) \geq n + 1
\end{equation*}
\noindent holds for all periodic Reeb orbits $\gamma$. Thus for $n \geq 2$, the index condition is satisfied and we actually compute $\check{\SH}_*(\mathbb{S}^{2n - 1}/\mathbb{Z}_m)$ via the $\mathbb{Z}_m$-equivariant version of $\check{\SH}_*(\mathbb{S}^{2n - 1})$.

In the case of hypertight contact manifolds, there is a similar construction without the index condition \eqref{eq:index_condition}. See for example \cite[Theorem~1.1]{meiwesnaef:hypertight:2018}. By \cite[Theorem~1.7]{meiwesnaef:hypertight:2018}, there do exist noncontractible periodic Reeb orbits on hypertight contact manifolds under suitable technical conditions. Moreover, one can show the existence of invariant Reeb orbits in this setting. See \cite[Corollary~1.6]{meiwesnaef:hypertight:2018} as well as \cite[Theorem~1.6]{merrynaef:invariant:2016} in the Liouville-fillable case.

\section{The Twisted Rabinowitz Action Functional}
\label{sec:the_twisted_Rabinowitz_action_functional}

\begin{definition}[Free Twisted Loop Space]
	Let $\varphi \in \Diff(M)$ be a diffeomorphism of a smooth manifold $M$. Define the \bld{free twisted loop space of $M$ and $\varphi$} by
	\begin{equation*}
		\mathscr{L}_\varphi M := \cbr[0]{\gamma \in C^\infty(\mathbb{R},M) : \gamma(t + 1) = \varphi(\gamma(t)) \> \forall t \in \mathbb{R}}.
	\end{equation*}
\end{definition}

Let $(M,\omega)$ be a symplectic manifold and $\varphi \in \Symp(M,\omega)$. Given a twisted loop $\gamma \in \mathscr{L}_\varphi M$ and $\varepsilon_0 > 0$, we say that a curve
\begin{equation*}
	\intoo[0]{-\varepsilon_0,\varepsilon_0} \to \mathscr{L}_\varphi M, \qquad \varepsilon \mapsto \gamma_\varepsilon
\end{equation*}
\noindent starting at $\gamma$ is \bld{smooth}, iff the induced variation
\begin{equation*}
	\mathbb{R} \times \intoo[0]{-\varepsilon_0,\varepsilon_0} \to M, \qquad (t,\varepsilon) \mapsto \gamma_\varepsilon(t)
\end{equation*}
\noindent is smooth. Since $\gamma_\varepsilon(t + 1) = \varphi(\gamma_\varepsilon(t))$ holds for all $\varepsilon \in \intoo[0]{-\varepsilon_0,\varepsilon_0}$ and $t \in \mathbb{R}$, we call such a variation a \bld{twisted variation}. Then the infinitesimal variation
\begin{equation*}
	\delta\gamma := \frac{\partial \gamma_\varepsilon}{\partial \varepsilon}\bigg\vert_{\varepsilon = 0} \in \mathfrak{X}(\gamma),
\end{equation*}
\noindent satisfies
\begin{equation*}
	\delta\gamma(t + 1) = D\varphi(\delta\gamma(t)) \qquad \forall t \in \mathbb{R}.
\end{equation*}

\begin{lemma}
	Let $(M,\omega)$ be a symplectic manifold and let $\varphi \in \Symp(M,\omega)$ be of finite order. Let $\gamma \in \mathscr{L}_\varphi M$ and let $X \in \mathfrak{X}(\gamma)$ be such that
	\begin{equation*}
		X(t + 1) = D\varphi(X(t)) \qquad \forall t \in \mathbb{R}.
	\end{equation*}
	Then there exists a twisted variation of $\gamma$ such that $\delta \gamma = X$.
	\label{lem:twisted_variation}
\end{lemma}

\begin{proof}
	As $\varphi$ is assumed to be of finite order, there exists a $\varphi$-invariant $\omega$-compatible almost complex structure $J$ on $M$ by \cite[Lemma~5.5.6]{mcduffsalamon:st:2017}. With respect to the induced Riemannian metric
	\begin{equation*}
		m_J := \omega(J\cdot,\cdot),
	\end{equation*}
	\noindent the symplectomorphism $\varphi$ is an isometry. Define the exponential variation
	\begin{equation*}
		\mathbb{R} \times \intoo[0]{-\varepsilon_0,\varepsilon_0} \to M, \qquad \gamma_\varepsilon(t) := \exp^{\nabla_J}_{\gamma(t)}(\varepsilon X(t)),
	\end{equation*}
	\noindent for $\varepsilon_0 > 0$ sufficiently small and $\nabla_J$ denoting the Levi--Civita connection associated with $m_J$. Such an $\varepsilon_0 > 0$ does exist by naturality of geodesics \cite[Corollary~5.14]{lee:dg:2018}. Then we compute
	\begin{align*}
		\gamma_\varepsilon(t + 1) &= \exp^{\nabla_J}_{\gamma(t + 1)}(\varepsilon X(t + 1))\\
		&= \exp^{\nabla_J}_{\varphi(\gamma(t))}(D\varphi(\varepsilon X(t)))\\
		&= \varphi \del[1]{\exp^{\nabla_J}_{\gamma(t)}(\varepsilon X(t))}\\
		&= \varphi(\gamma_\varepsilon(t))
	\end{align*}
	\noindent by naturality of the exponential map \cite[Proposition~5.20]{lee:dg:2018}.
\end{proof}

\begin{remark}
	The statement of Lemma \ref{lem:twisted_variation} remains true if $\ord \varphi = \infty$.
\end{remark}

This discussion together with Lemma \ref{lem:free_twisted_loop_space} motivates the following definition of the tangent space to the free twisted loop space.

\begin{definition}[Tangent Space to the Free Twisted Loop Space]
	Let $(M,\omega)$ be a symplectic manifold and $\varphi \in \Symp(M,\omega)$. For $\gamma \in \mathscr{L}_\varphi M$ define the \bld{tangent space to the free twisted loop space at $\gamma$} by
	\begin{equation*}
		T_\gamma\mathscr{L}_\varphi M := \cbr[0]{X \in \Gamma(\gamma^*TM) : X(t + 1) = D\varphi(X(t)) \> \forall t \in \mathbb{R}}.
	\end{equation*}
\end{definition}

\begin{definition}[Twisted Hamiltonian Function]
	Let $(M,\omega)$ be a symplectic manifold and $\varphi \in \Symp(M,\omega)$. A function $H \in C^\infty(M \times \mathbb{R})$ is said to be a \bld{twisted Hamiltonian function}, iff
	\begin{equation*}
		\varphi^*H_{t + 1} = H_t \qquad \forall t \in \mathbb{R}.
	\end{equation*}
	We denote the space of all twisted Hamiltonian functions by $C^\infty_\varphi(M \times \mathbb{R})$ and the subspace of all autonomous twisted Hamiltonian functions by $C^\infty_\varphi(M)$.
\end{definition}

Recall, that an exact symplectic manifold is by definition a pair $(M,\lambda)$ such that $(M,d\lambda)$ is a symplectic manifold. An exact symplectomorphism of an exact symplectic manifold $(M,\lambda)$ is a diffeomorphism $\varphi \in \Diff(M)$ such that $\varphi^*\lambda - \lambda$ is exact.

\begin{definition}[Perturbed Twisted Rabinowitz Action Functional]
	Let $(M,\lambda)$ be an exact symplectic manifold and $\varphi \in \Diff(M)$ an exact symplectomorphism with $\varphi^*\lambda - \lambda = df$. For $H,F \in C^\infty_\varphi(M \times \mathbb{R})$ define the \bld{perturbed twisted Rabinowitz action functional}
	\begin{equation*}
		\mathscr{A}^{(H,F)}_\varphi \colon \mathscr{L}_\varphi M \times \mathbb{R} \to \mathbb{R}
	\end{equation*}
	\noindent by
	\begin{equation*}
		\mathscr{A}^{(H,F)}_\varphi(\gamma,\tau) := \int_0^1 \gamma^*\lambda - \tau \int_0^1 H_t(\gamma(t))dt - \int_0^1 F_t(\gamma(t))dt - f(\gamma(0)).
	\end{equation*}
	If $F = 0$ and $H \in C^\infty_\varphi(M)$, we write $\mathscr{A}^H_\varphi$ for $\mathscr{A}^{(H,F)}_\varphi$ and call $\mathscr{A}^H_\varphi$ the \bld{twisted Rabinowitz action functional}.
	\label{def:perturbed_twisted_Rabinowitz_functional}
\end{definition}

\begin{remark}
	Assume that $m := \ord \varphi < \infty$. Then
	\begin{equation*}
		\mathscr{A}^{(H,F)}_\varphi(\gamma,\tau) = \frac{1}{m} \mathscr{A}^{(H,F)}(\bar{\gamma},\tau) - \frac{1}{m} \sum_{k = 0}^{m - 1}f(\gamma(k)),
	\end{equation*}
	\noindent for all $(\gamma,\tau) \in \mathscr{L}_\varphi M$, where $\bar{\gamma} \in \mathscr{L} M$ is defined by $\bar{\gamma}(t) := \gamma(mt)$ and
	\begin{equation*}
		\mathscr{A}^{(H,F)} \colon \mathscr{L} M \times \mathbb{R} \to \mathbb{R}
	\end{equation*}
	\noindent denotes the standard Rabinowitz action functional.
	\label{rem:finite_order}
\end{remark}

\begin{definition}[Differential of the Perturbed Twisted Rabinowitz Action Functional]
	Let $\varphi \in \Diff(M)$ be an exact symplectomorphism of an exact symplectic manifold $(M,\lambda)$. For $H,F \in C^\infty_\varphi(M \times \mathbb{R})$, define the \bld{differential of the perturbed twisted Rabinowitz action functional}
	\begin{equation*}
		d\mathscr{A}^{(H,F)}_\varphi\vert_{(\gamma,\tau)} \colon T_\gamma \mathscr{L}_\varphi M \times \mathbb{R} \to \mathbb{R}
	\end{equation*}
	\noindent for all $(\gamma,\tau) \in \mathscr{L}_\varphi M \times \mathbb{R}$ by
	\begin{equation*}
		d\mathscr{A}^{(H,F)}_\varphi\vert_{(\gamma,\tau)}(X,\eta) := \frac{d}{d\varepsilon}\bigg\vert_{\varepsilon = 0}\mathscr{A}^{(H,F)}_\varphi(\gamma_\varepsilon, \tau + \varepsilon \eta),
	\end{equation*}
	\noindent where $\gamma_\varepsilon$ is a twisted variation of $\gamma$ such that $\delta \gamma = X$.
\end{definition}

\begin{proposition}[Differential of the Perturbed Twisted Rabinowitz Action Functional]
	Let $\varphi \in \Diff(M)$ be an exact symplectomorphism of an exact symplectic manifold $(M,\lambda)$ and $H,F \in C^\infty_\varphi(M \times \mathbb{R})$. Then
	\begin{multline}
		d\mathcal{\mathscr{A}}^{(H,F)}_\varphi \vert_{(\gamma,\tau)}(X,\eta) = \int_0^1 d\lambda(X(t),\dot{\gamma}(t) - \tau X_{H_t}(\gamma(t)) -  X_{F_t}(\gamma(t)))dt\\
		- \eta\int_0^1 H_t(\gamma(t))dt
		\label{eq:differential_perturbed_twisted_Rabinowitz_functional}
	\end{multline}
	\noindent for all $(\gamma,\tau) \in \mathscr{L}_\varphi M \times \mathbb{R}$ and $(X,\eta) \in T_\gamma\mathscr{L}_\varphi M \times \mathbb{R}$. Moreover, we have that
	\begin{equation*}
		(\gamma,\tau) \in \Crit \mathscr{A}^{(H,F)}_\varphi
	\end{equation*}
	\noindent if and only if
	\begin{equation}
		\dot{\gamma}(t) = \tau X_{H_t}(\gamma(t)) + X_{F_t}(\gamma(t)) \qquad \text{and} \qquad \int_0^1 H_t(\gamma(t))dt = 0
		\label{eq:critical_points}
	\end{equation}
	\noindent for all $t \in \mathbb{R}$.
	\label{prop:differential_perturbed_twisted_Rabinowitz_functional}
\end{proposition}

\begin{proof}
	A routine computation shows \eqref{eq:differential_perturbed_twisted_Rabinowitz_functional}. Let $(\gamma,\tau) \in \Crit \mathscr{A}^{(H,F)}_\varphi$. It follows immediately from \eqref{eq:differential_perturbed_twisted_Rabinowitz_functional} that
	\begin{equation*}
		\int_0^1 H_t(\gamma(t))dt = 0
	\end{equation*}
	\noindent and
	\begin{equation*}
		\int_0^1 d\lambda(X(t),\dot{\gamma}(t) - \tau X_{H_t}(\gamma(t)) -  X_{F_t}(\gamma(t)))dt = 0
	\end{equation*}
	\noindent for all $X \in T_\gamma \mathscr{L}_\varphi M$. Suppose there exists $t_0 \in \Int I$ such that
	\begin{equation*}
		\dot{\gamma}(t_0) - \tau X_{H_{t_0}}(\gamma(t_0)) -  X_{F_{t_0}}(\gamma(t_0)) \neq 0.
	\end{equation*}
	By nondegeneracy of the symplectic form $d\lambda$ there exists $v \in T_{\gamma(t_0)}M$ with
	\begin{equation*}
		d\lambda(v,\dot{\gamma}(t_0) - \tau X_{H_{t_0}}(\gamma(t_0)) -  X_{F_{t_0}}(\gamma(t_0))) \neq 0.
	\end{equation*}
	Fix a Riemannian metric on $M$ and let $X_v$ denote the unique parallel vector field along $\gamma \vert_I$ such that $X_v(t_0) = v$. As $\Int I$ is open, there exists $\delta > 0$ such that $\bar{B}_\delta(t_0) \subseteq \Int I$. Fix a smooth bump function $\beta \in C^\infty(I)$ for $t_0$ supported in $B_\delta(t_0)$. By shrinking $\delta$ if necessary, we may assume that
	\begin{equation*}
		\int_{t_0 - \delta}^{t_0 + \delta} d\lambda(\beta(t)X_v(t),\dot{\gamma}(t) - \tau X_{H_t}(\gamma(t)) -  X_{F_t}(\gamma(t)))dt \neq 0.
	\end{equation*}
	Extending
	\begin{equation*}
		(\beta X_v)(t + k) := D\varphi^k(\beta(t)X_v(t)) \qquad \forall t \in I, k \in \mathbb{Z},
	\end{equation*}
	\noindent we have that $\beta X_v \in T_\gamma \mathscr{L}_\varphi M$ and thus we compute
	\begin{align*}
		0 &= d\mathcal{\mathscr{A}}^{(H,F)}_\varphi \vert_{(\gamma,\tau)}(\beta X_v,0)\\
		&= \int_{t_0 - \delta}^{t_0 + \delta} d\lambda(\beta(t)X_v(t),\dot{\gamma}(t) - \tau X_{H_t}(\gamma(t)) -  X_{F_t}(\gamma(t)))dt\\
		&\neq 0.
	\end{align*}
	Hence
	\begin{equation*}
		\dot{\gamma}(t) = \tau X_{H_t}(\gamma(t)) + X_{F_t}(\gamma(t)) \qquad \forall t \in I,
	\end{equation*}
	\noindent implying
	\begin{align*}
		\dot{\gamma}(t + k) &= D\varphi^k(\dot{\gamma}(t))\\
		&= \tau (D\varphi^k \circ X_{H_t})(\gamma(t)) + (D\varphi^k \circ X_{F_t})(\gamma(t))\\
		&= \tau (D\varphi^k \circ X_{H_t} \circ \varphi^{-k} \circ \varphi^k)(\gamma(t)) + (D\varphi^k \circ X_{F_t} \circ \varphi^{-k} \circ \varphi^k)(\gamma(t))\\
		&= \tau \varphi^k_* X_{H_t}(\gamma(t + k)) + \varphi^k_* X_{F_t}(\gamma(t + k))\\
		&= \tau X_{\varphi^k_*H_t}(\gamma(t + k)) + X_{\varphi^k_*F_t}(\gamma(t + k))\\
		&= \tau X_{H_{t + k}}(\gamma(t + k)) + X_{F_{t + k}}(\gamma(t + k))
	\end{align*}
	\noindent for all $t \in I$ and $k \in \mathbb{Z}$. The other direction is immediate.
\end{proof}

\begin{corollary}
	The differential of the perturbed twisted Rabinowitz action functional is well-defined, that is, independent of the choice of twisted variation, and linear.
\end{corollary}

Preservation of energy of an autonomous Hamiltonian system yields the following corollary.

\begin{corollary}
	Let $\varphi \in \Diff(M)$ be an exact symplectomorphism of an exact symplectic manifold $(M,\lambda)$ and $H \in C^\infty_\varphi(M)$. Then $\Crit\mathscr{A}^H_\varphi$ consists precisely of all $(\gamma,\tau) \in \mathscr{L}_\varphi M \times \mathbb{R}$ such that $\gamma(\mathbb{R}) \subseteq H^{-1}(0)$ and $\gamma$ is an integral curve of $\tau X_H$.
	\label{cor:critical_points_perturbed_twisted_Rabinowitz_functional}
\end{corollary}

There is a natural $\mathbb{R}$-action on the twisted loop space $\mathscr{L}_\varphi M$ given by
\begin{equation*}
	(s \cdot \gamma)(t) := \gamma(t + s) \qquad \forall t \in \mathbb{R}.
\end{equation*}
If $(M,\lambda)$ is an exact symplectic manifold and $H \in C^\infty_\varphi(M)$ for an exact symplectomorphism $\varphi \in \Diff(M)$ of finite order such that $\supp f \cap H^{-1}(0) = \emptyset$, then the twisted Rabinowitz action functional $\mathscr{A}^H_\varphi$ is invariant under the induced $\mathbb{S}^1$-action on $\Crit \mathscr{A}^H_\varphi$. In particular, the unperturbed twisted Rabinowitz action functional is never a Morse function.

\begin{definition}[Hessian of the Twisted Rabinowitz Action Functional]
	Let $\varphi \in \Diff(M)$ be an exact symplectomorphism of an exact symplectic manifold $(M,\lambda)$ and $H \in C^\infty_\varphi(M)$. For $(\gamma,\tau) \in \Crit \mathscr{A}^H_\varphi$, define the \bld{Hessian of the twisted Rabinowitz action functional} 
	\begin{equation*}
		\Hess \mathscr{A}^H_\varphi\vert_{(\gamma,\tau)} \colon (T_\gamma\mathscr{L}_\varphi M \times \mathbb{R}) \times (T_\gamma \mathscr{L}_\varphi M \times \mathbb{R}) \to \mathbb{R}
	\end{equation*}
	\noindent by
	\begin{equation*}
		\Hess \mathscr{A}^H_\varphi\vert_{(\gamma,\tau)}((X,\eta),(Y,\sigma)) := \frac{\partial^2}{\partial\varepsilon_1\partial\varepsilon_2}\bigg\vert_{\varepsilon_1 = \varepsilon_2 = 0}\mathscr{A}^H_\varphi(\gamma_{\varepsilon_1,\varepsilon_2},\tau + \varepsilon_1\eta + \varepsilon_2\sigma),
	\end{equation*}
	\noindent for a smooth two-parameter family $\gamma_{\varepsilon_1,\varepsilon_2}$ of twisted loops with
	\begin{equation*}
		\frac{\partial}{\partial \varepsilon_1}\bigg\vert_{\varepsilon_1 = 0}\gamma_{\varepsilon_1,0} = X \qquad \text{and} \qquad \frac{\partial}{\partial \varepsilon_2}\bigg\vert_{\varepsilon_2 = 0}\gamma_{0,\varepsilon_2} = Y.
	\end{equation*}
\end{definition}

\begin{definition}[Symplectic Connection]
	Let $(M,\omega)$ be a symplectic manifold. A \bld{symplectic connection on $(M,\omega)$} is defined to be a torsion-free connection $\nabla$ in the tangent bundle $TM$ such that $\nabla \omega = 0$.
\end{definition}

\begin{remark}
	Every symplectic manifold admits a symplectic connection by \cite[p.~308]{gutt:symplectic:2006}, but in sharp contrast to the Riemannian case, a symplectic connection on a given symplectic manifold is in general not unique.
\end{remark}

\begin{lemma}
	Let $\varphi \in \Diff(M)$ be an exact symplectomorphism of an exact symplectic manifold $(M,\lambda)$. Fix a symplectic connection $\nabla$ on $(M,d\lambda)$ and a twisted Hamiltonian function $H \in C^\infty_\varphi(M)$. If $(\gamma,\tau) \in \Crit \mathscr{A}^H_\varphi$, then
	\begin{multline}
		\Hess \mathscr{A}^H_\varphi\vert_{(\gamma,\tau)}((X,\eta),(Y,\sigma)) = \int_0^1 d\lambda(Y,\nabla_t X)\\ - \tau\int_0^1 \Hess^\nabla H(X,Y) - \eta\int_0^1 dH(Y) - \sigma \int_0^1 dH(X)
		\label{eq:formula_twisted_Hessian}
	\end{multline}
	\noindent for all $(X,\eta),(Y,\sigma) \in T_\gamma \mathscr{L}_\varphi M \times \mathbb{R}$.
	\label{lem:twisted_Hessian}
\end{lemma}

\begin{proof}
	The proof is a long routine computation.
\end{proof}

\begin{corollary}
	The Hessian of the twisted Rabinowitz action functional is a well-defined, that is, independent of the choice of twisted two-parameter family, symmetric bilinear form.
\end{corollary}

In fact, the Hessian of the twisted Rabinowitz action functional is also independent of the choice of symplectic connection.

\begin{lemma}
	Let $\varphi \in \Diff(M)$ be an exact symplectomorphism of an exact symplectic manifold $(M,\lambda)$ and $H \in C^\infty_\varphi(M)$. If $(\gamma,\tau) \in \Crit \mathscr{A}^H_\varphi$, then
	\begin{multline}
		\label{eq:formula_twisted_Hessian_Lie}
		\Hess \mathscr{A}^H_\varphi\vert_{(\gamma,\tau)}((X,\eta),(Y,\sigma)) = \int_0^1 d\lambda(Y,L_{\tau X_H}X - \eta X_H(\gamma))\\ - \sigma\int_0^1 dH(X)
	\end{multline}
	\noindent for all $(X,\eta),(Y,\sigma) \in T_\gamma \mathscr{L}_\varphi M \times \mathbb{R}$, where
	\begin{equation*}
		L_{\tau X_H}X(t) = \frac{d}{ds}\bigg\vert_{s = 0} D\phi^H_{-s\tau}X(s + t) \qquad \forall t \in I,
	\end{equation*}
	\noindent with $\phi^H$ denoting the smooth flow of the Hamiltonian vector field $X_H$.
	\label{lem:twisted_Hessian}
\end{lemma}

\begin{proof}
	One computes
	\begin{equation*}
		\Hess^\nabla(X,Y) = d\lambda(Y,\nabla_X X_H).
	\end{equation*}
	Inserting this into \eqref{eq:formula_twisted_Hessian} yields
	\begin{multline*}
		\Hess \mathscr{A}^H_\varphi\vert_{(\gamma,\tau)}((X,\eta),(Y,\sigma)) = \int_0^1 d\lambda(Y,\nabla_t X - \tau\nabla_X X_H)\\ - \eta \int_0^1 dH(Y) - \sigma \int_0^1 dH(X).
	\end{multline*}
	But as $\nabla$ has no torsion by assumption, we compute
	\begin{equation*}
		\nabla_t X - \tau \nabla_X X_H = \nabla_{\dot{\gamma}}X - \tau \nabla_X X_H = \nabla_{\tau X_H}X - \tau\nabla_X X_H = [\tau X_H,X],
	\end{equation*}
	\noindent and 
	\allowdisplaybreaks
	\begin{align*}
		[\tau X_H,X](t) &= L_{\tau X_H}X(t)\\
		&= \frac{d}{ds}\bigg\vert_{s = 0} D\phi_{-s\tau}^H(X(\phi_{s\tau}^H(\gamma(t)))\\
		&= \frac{d}{ds}\bigg\vert_{s = 0} D\phi_{-s\tau}^H(X(\phi_{s\tau}^H(\phi_{t\tau}^H(\gamma(0)))))\\
		&= \frac{d}{ds}\bigg\vert_{s = 0} D\phi_{-s\tau}^H(X(\phi_{(s + t)\tau}^H(\gamma(0))))\\
		&= \frac{d}{ds}\bigg\vert_{s = 0} D\phi_{-s\tau}^H X(s + t)
	\end{align*}
	\noindent for all $t \in I$.
\end{proof}

\begin{corollary}
	\label{cor:kernel_Hessian}
	Let $\varphi \in \Diff(M)$ be an exact symplectomorphism of an exact symplectic manifold $(M,\lambda)$ and $H \in C^\infty_\varphi(M)$. The kernel of the Hessian of the twisted Rabinowitz action functional at $(\gamma,\tau) \in \Crit \mathscr{A}^H_\varphi$ consists precisely of all $(X,\eta) \in T_\gamma \mathscr{L}_\varphi M \times \mathbb{R}$ satisfying
	\begin{equation*}
		L_{\tau X_H} X = \eta X_H(\gamma) \qquad \text{and} \qquad \int_0^1 dH(X) = 0.
	\end{equation*}
\end{corollary}


\begin{lemma}
	Let $\varphi \in \Diff(M)$ be an exact symplectomorphism of an exact symplectic manifold $(M,\lambda)$ and $H \in C^\infty_\varphi(M)$. For every $(\gamma,\tau) \in \Crit \mathscr{A}^H_\varphi$, there is a canonical isomorphism
	\begin{equation}
		\ker \Hess \mathscr{A}^H_\varphi\vert_{(\gamma,\tau)} \cong \mathfrak{K}(\gamma,\tau),
		\label{eq:canonical_isomorphism}
	\end{equation}
	\noindent where
	\begin{equation*}
		\mathfrak{K}(\gamma,\tau) := \cbr[0]{(v_0,\eta) \in T_{\gamma(0)}M \times \mathbb{R} : \textup{solution of \eqref{eq:initial_value_system}}}
	\end{equation*}
	\noindent with
	\begin{equation}
		\label{eq:initial_value_system}
		D(\phi_{-\tau}^{X_H} \circ \varphi)v_0 = v_0 + \eta X_H(\gamma(0)) \qquad \text{and} \qquad dH(v_0) = 0.
	\end{equation}
	\label{lem:kernel_of_the_Hessian}
\end{lemma}

\begin{proof}
	We follow \cite[p.~99--100]{frauenfelderkoert:3bp:2018}. Let $(X,\eta) \in \ker \Hess \mathscr{A}^H_\varphi\vert_{(\gamma,\tau)}$ and define
	\begin{equation*}
		v \colon I \to T_{\gamma(0)}M, \qquad v(t) := D\phi_{-\tau t}^H X(t).
	\end{equation*}
	We claim that
	\begin{equation}
		\ker \Hess \mathscr{A}^H_\varphi\vert_{(\gamma,\tau)} \to \mathfrak{K}(\gamma,\tau), \qquad (X,\eta) \mapsto (v(0),\eta)
		\label{eq:canonical_isomorphism}
	\end{equation}
	\noindent is an isomorphism. First, we show that the above homomorphism is indeed well-defined. The assumption that $(X,\eta)$ lies in the kernel of the Hessian of the twisted Rabinowitz action functional at the critical point $(\gamma,\tau)$ is by Corollary \ref{cor:kernel_Hessian} equivalent to
	\begin{equation}
		\label{eq:kernel_of_the_Hessian}
		\dot{v} = \eta X_H(\gamma(0)) \qquad \text{and} \qquad \int_0^1 dH(v) = 0.
	\end{equation}
	Integrating the first equation yields
	\begin{equation*}
		v(t) = v_0 + t\eta X_H(\gamma(0)) \qquad \forall t \in I,
	\end{equation*}
	\noindent with $v_0 :=v(0)$. Thus $(v_0,\eta) \in \mathfrak{K}(\gamma,\tau)$ follows from
	\begin{align}
		\label{eq:twist_condition}
		v(1) &= D\phi_{-\tau}^H X(1)\nonumber\\
		&= D\phi_{-\tau}^H D\varphi(X(0))\nonumber\\
		&= D(\phi_{-\tau}^H \circ \varphi)X(0)\nonumber\\
		&= D(\phi_{-\tau}^H \circ \varphi)v_0.
	\end{align}
	That \eqref{eq:canonical_isomorphism} is an isomorphism follows by considering the inverse
	\begin{equation*}
		\mathfrak{K}(\gamma,\tau) \to \ker \Hess \mathscr{A}^H_\varphi\vert_{(\gamma,\tau)}, \qquad (v_0,\eta) \mapsto (X,\eta),
	\end{equation*}
	\noindent where $X \in T_\gamma \mathscr{L}_\varphi M$ is defined by
	\begin{equation*}
		X(t) := D\phi^H_{\tau t}(v_0 + t\eta X_H(\gamma(0))) \qquad \forall t \in \mathbb{R}.
	\end{equation*}
	This establishes the canonical isomorphism \eqref{eq:canonical_isomorphism}.
\end{proof}

In what follows, we assume that the energy hypersurface $H^{-1}(0)$ is a contact manifold. A contact manifold is a pair $(\Sigma,\alpha)$, where $\Sigma$ is an odd-dimensional manifold and $\alpha \in \Omega^1(\Sigma)$ is a global contact form. Every contact manifold $(\Sigma,\alpha)$ admits a unique vector field $R \in \mathfrak{X}(\Sigma)$, called the Reeb vector field, defined implicitly by
\begin{equation*}
	i_Rd\alpha = 0 \qquad \text{and} \qquad i_R\alpha = 1.
\end{equation*}
Recall, that a strict contactomorphism of a contact manifold $(\Sigma,\alpha)$ is defined to be a diffeomorphism $\varphi \in \Diff(\Sigma)$ such that $\varphi^*\alpha = \alpha$. Note that the Reeb flow always commutes with a strict contactomorphism.

\begin{definition}[Parametrised Twisted Reeb Orbit]
	For a contact manifold $(\Sigma,\alpha)$ and a strict contactomorphism $\varphi \colon (\Sigma,\alpha) \to (\Sigma,\alpha)$ define the set of \bld{parametrised twisted Reeb orbits on $(\Sigma,\alpha)$} by
	\begin{equation*}
		\mathscr{P}_\varphi(\Sigma,\alpha) := \cbr[0]{(\gamma,\tau) \in \mathscr{L}_\varphi \Sigma \times \mathbb{R} : \dot{\gamma}(t) = \tau R(\gamma(t)) \> \forall t \in \mathbb{R}}.
	\end{equation*}
\end{definition}

\begin{definition}[Twisted Spectrum]
	For a contact manifold $(\Sigma,\alpha)$ and a strict contactomorphism $\varphi \colon (\Sigma,\alpha) \to (\Sigma,\alpha)$ define the \bld{twisted spectrum} by
		\begin{equation*}
		\Spec(\Sigma,\alpha) := \cbr[0]{\tau \in \mathbb{R} : \exists \gamma \in \mathscr{L}_\varphi \Sigma \> \text{such that } (\gamma,\tau) \in \mathscr{P}_\varphi(\Sigma,\alpha)}.
	\end{equation*}
\end{definition}



\begin{proposition}[Kernel of the Hessian of the Twisted Rabinowitz Action Functional]
	Let $(\Sigma,\lambda\vert_\Sigma)$ be a regular energy surface of restricted contact type in an exact Hamiltonian system $(M,\lambda,H)$ with $X_H\vert_\Sigma = R$. Suppose $\varphi \in \Diff(M)$ is an exact symplectomorphism such that $H \in C^\infty_\varphi(M)$ and $\varphi^*\lambda\vert_\Sigma = \lambda\vert_\Sigma$. Then
	\begin{equation*}
		\Crit \mathscr{A}^H_\varphi = \mathscr{P}_\varphi(\Sigma,\lambda\vert_\Sigma)
	\end{equation*}
	\noindent and
	\begin{equation*}
		\ker \Hess \mathscr{A}^H_\varphi\vert_{(\gamma,\tau)} \cong \ker \del[1]{D(\phi^R_{-\tau} \circ \varphi)\vert_{\gamma(0)} - \id_{T_{\gamma(0)}\Sigma}}
	\end{equation*}
	\noindent for all $ (\gamma,\tau) \in \mathscr{P}_\varphi(\Sigma,\lambda\vert_\Sigma)$. Moreover, we have $R(\gamma(0)) \in \ker \Hess \mathscr{A}^H_\varphi\vert_{(\gamma,\tau)}$ and if $\mathscr{P}_\varphi(\Sigma,\lambda\vert_\Sigma) \subseteq \Sigma \times \mathbb{R}$ is an embedded submanifold, then $\Spec(\Sigma,\lambda\vert_\Sigma)$ is discrete.
	\label{prop:kernel_hessian_contact}
\end{proposition}

\begin{remark}
	\label{rem:period-action_equality}
	If $(\gamma,\tau) \in \mathscr{P}_\varphi(\Sigma,\lambda\vert_\Sigma)$, we have the period-action equality
	\begin{equation*}
		\mathscr{A}^H_\varphi(\gamma,\tau) = \int_0^1 \gamma^* \lambda = \int_0^1 \lambda(\dot{\gamma}) = \tau \int_0^1 \lambda(R(\gamma)) = \tau.
	\end{equation*}	
\end{remark}

\begin{proof}
	The identity $\Crit \mathscr{A}^H_\varphi = \mathscr{P}_\varphi(\Sigma,\lambda\vert_\Sigma)$ immediately follows from Corollary \ref{cor:critical_points_perturbed_twisted_Rabinowitz_functional} together with \cite[Corollary~5.30]{lee:dt:2012}. The proof of the formula for the kernel of the Hessian of $\mathscr{A}^H_\varphi$ is inspired by \cite[p.~102]{frauenfelderkoert:3bp:2018}. By Lemma \ref{lem:kernel_of_the_Hessian} we have that
	\begin{equation*}
		\ker \Hess \mathscr{A}^H_\varphi\vert_{(\gamma,\tau)} \cong \mathfrak{K}(\gamma,\tau),
	\end{equation*}
	\noindent where $(v_0,\eta) \in T_{\gamma(0)} M \times \mathbb{R}$ belongs to $\mathfrak{K}(\gamma,\tau)$ if and only if
	\begin{equation*}
		D(\phi^{X_H}_{-\tau} \circ \varphi)v_0 = v_0 + \eta X_H(\gamma(0)) \qquad \text{and} \qquad dH(v_0) = 0.
	\end{equation*}
	Thus in our setting, the second condition implies $v_0 \in T_{\gamma(0)}\Sigma$. Decompose
	\begin{equation*}
		v_0 = v^\xi_0 + a R(\gamma(0)) \qquad v^\xi_0 \in \xi_{\gamma(0)}, a \in \mathbb{R},
	\end{equation*}
	\noindent where $\xi := \ker \lambda\vert_\Sigma$ denotes the contact distribution. Then we compute
	\begin{align*}
		D(\phi^R_{-\tau} \circ \varphi)R(\gamma(0)) &= D(\phi^R_{-\tau} \circ \varphi)\del[3]{\frac{d}{dt}\bigg\vert_{t = 0} \phi^R_t(\gamma(0))}\\
		&= \frac{d}{dt}\bigg\vert_{t = 0} (\phi^R_{-\tau} \circ \varphi \circ \phi^R_t)(\gamma(0))\\
		&= \frac{d}{dt}\bigg\vert_{t = 0} (\phi^R_t \circ \varphi \circ \phi^R_{-\tau})(\gamma(0))\\
		&= \frac{d}{dt}\bigg\vert_{t = 0} \phi^R_t(\gamma(0))\\
		&= R(\gamma(0)),
	\end{align*}
	\noindent as a strict contactomorphism commutes with the Reeb flow. Hence
	\begin{equation*}
		v_0 + \eta R(\gamma(0)) = D(\phi^R_{-\tau} \circ \varphi)v_0 = D^\xi(\phi^R_{-\tau} \circ \varphi)v_0^\xi + aR(\gamma(0)),
	\end{equation*}
	\noindent where 
	\begin{equation*}
		D^\xi(\phi^R_{-\tau} \circ \varphi) := D(\phi^R_{-\tau} \circ \varphi)\vert_\xi \colon \xi \to \xi,
	\end{equation*}
	\noindent implies 
	\begin{equation*}
		\eta = 0 \qquad \text{and} \qquad D^\xi(\phi^R_{-\tau} \circ \varphi)v_0^\xi = v_0^\xi
	\end{equation*}
	\noindent by considering the splitting $T\Sigma = \xi \oplus \langle R \rangle$. Consequently
	\begin{equation*}
		\mathfrak{K}(\gamma,\tau) = \ker \del[1]{D(\phi^R_{-\tau} \circ \varphi)\vert_{\gamma(0)} - \id_{T_{\gamma(0)}\Sigma}} \times \cbr[0]{0}.
	\end{equation*}

	Finally, assume that $\mathscr{P}_\varphi(\Sigma,\lambda\vert_\Sigma) \subseteq \Sigma \times \mathbb{R}$ is an embedded submanifold via the obvious identification of $(\gamma,\tau) \in \mathscr{P}_\varphi(\Sigma,\lambda\vert_\Sigma)$ with $(\gamma(0),\tau) \in \Sigma \times \mathbb{R}$. Fix a path $(\gamma_s,\tau_s)$ in $\mathscr{P}_\varphi(\Sigma,\lambda\vert_\Sigma) = \Crit \mathscr{A}^H_\varphi$ from $(\gamma_0,\tau_0)$ to $(\gamma_1,\tau_1)$. Then using Remark \ref{rem:period-action_equality} we compute
	\begin{equation*}
		\partial_s \tau_s = \partial_s \mathscr{A}^H_\varphi(\gamma_s,\tau_s) = d\mathscr{A}^H_\varphi\vert_{(\gamma_s,\tau_s)}(\partial_s\gamma_s,\partial_s \tau_s) = 0,
	\end{equation*}
	\noindent implying that $\tau_s$ is constant, and in particular $\tau_0 = \tau_1$. Consequently, $\mathscr{A}^H_\varphi$ is constant on each path-connected component of $\mathscr{P}_\varphi(\Sigma,\lambda\vert_\Sigma)$. As $\mathscr{P}_\varphi(\Sigma,\lambda\vert_\Sigma)$ is a submanifold of $\Sigma \times \mathbb{R}$, there are only countably many connected components by definition, implying that $\Spec(\Sigma,\lambda\vert_\Sigma)$ is discrete.
\end{proof}

\section{Compactness of the Moduli Space of Twisted Negative Gradient Flow Lines}
\label{sec:compactness}

\begin{definition}[Liouville Domain]
	A \bld{Liouville domain} is defined to be a compact connected exact symplectic manifold $(W,\lambda)$ with connected boundary such that the Liouville vector field $X$ defined implicitly by $i_Xd\lambda = \lambda$ is outward-pointing along the boundary. 
\end{definition}

\begin{definition}[Liouville Automorphism]
	Let $(W, \lambda)$ be a Liouville domain with boundary $\Sigma$. A diffeomorphism $\varphi \in \Diff(W)$ is said to be a \bld{Liouville automorphism}, iff $\varphi(\Sigma) = \Sigma$, $\varphi^*\lambda - \lambda$ is exact and compactly supported in $\Int W$, and $\ord \varphi < \infty$ near the boundary. The set of all Liouville automorphisms on the Liouville domain $(W,\lambda)$ is denoted by $\Aut(W,\lambda)$.
\end{definition}

\begin{remark}
	Let $\varphi \in \Aut(W,\lambda)$ be a Liouville automorphism. Then there exists a unique function $f_\varphi \in C^\infty_c(\Int W)$ such that
		\begin{equation*}
			\varphi^* \lambda - \lambda = df_\varphi.
		\end{equation*}
\end{remark}

\begin{remark}
	The set $\Aut(W,\lambda)$ of Liouville automorphisms of a Liouville domain $(W,\lambda)$ is in general not a group. Indeed, for $\varphi,\psi \in \Aut(W,\lambda)$ it is not necessarily true that $\varphi \circ \psi$ is of finite order unless $\varphi$ and $\psi$ commute.
\end{remark}

\begin{definition}[Completion of a Liouville Domain]
	Let $(W,\lambda)$ be a Liouville domain with boundary $\Sigma$. The \bld{completion of $(W,\lambda)$} is defined to be the exact symplectic manifold $(M,\lambda)$, where 
\begin{equation*}
	M := W \cup_\Sigma \intco[0]{0,+\infty} \times \Sigma \qquad \text{and} \qquad \lambda\vert_{\intco[0]{0,+\infty} \times \Sigma} := e^r\lambda\vert_\Sigma.
	\end{equation*}
\end{definition}

\begin{definition}[Twisted Defining Hamiltonian Function]
	Let $(W,\lambda)$ be a Liouville domain with boundary $\Sigma$ and $\varphi \in \Aut(W,\lambda)$. A \bld{twisted defining Hamiltonian function for $\Sigma$} is a Hamiltonian function $H \in C^\infty(M)$ on the completion $(M,\lambda)$ of $(W,\lambda)$, satisfying the following conditions:
	\begin{enumerate}[label=\textup{(\roman*)}]
		\item $H^{-1}(0) = \Sigma$ and $\Sigma \cap \Crit H = \varnothing$.
		\item $H \in C^\infty_\varphi(M)$. 
		\item $dH$ is compactly supported.
		\item $X_H\vert_\Sigma = R$ is the Reeb vector field of the contact form $\lambda\vert_\Sigma$.
	\end{enumerate}
	Denote by $\mathscr{F}_\varphi(\Sigma)$ the set of twisted defining Hamiltonian functions for $\Sigma$.
\end{definition}

\begin{remark}
	A necessary condition for $\mathscr{F}_\varphi(\Sigma) \neq \emptyset$ is that $\varphi^*R = R$. This is not true in general if $\varphi$  does not induce a strict contactomorphism on $\Sigma$.
\end{remark}

\begin{definition}[Adapted Almost Complex Structure]
	Let $(W,\lambda)$ be a Liouville domain with boundary $\Sigma$. An \bld{adapted almost complex structure} is defined to be a $d\lambda$-compatible almost complex structure $J$ on $(W,\lambda)$ such that $J$ restricts to define a $d\lambda\vert_\Sigma$-compatible almost complex structure on the contact distribution $\ker \lambda\vert_\Sigma$ and $JR = \partial_r$ holds near the boundary.
\end{definition}

\begin{definition}[Rabinowitz--Floer Data]
	Let $(M,\lambda)$ be the completion of a Liouville domain $(W,\lambda)$ with boundary $\Sigma$ and $\varphi \in \Aut(W,\lambda)$. \bld{Rabinowitz--Floer data for $\varphi$} is defined to be a pair $(H,J)$ consisting of a twisted defining Hamiltonian function $H \in \mathscr{F}_\varphi(\Sigma)$ for $\Sigma$ and of a smooth family $J = (J_t)_{t \in \mathbb{R}}$ of adapted almost complex structures on $W$ such that 
	\begin{equation*}
		\varphi^*J_{t + 1} = J_t \qquad \forall t \in \mathbb{R}.
	\end{equation*}
\end{definition}

\begin{remark}
	For simplicity we ignore the fact, that in order to achieve transversality of the moduli spaces in general one actually needs a dependence of the smooth family of almost complex structures on the Lagrange multiplier (see \cite{abbondandolomerry:floer:2018}). But this technicality does not significantly alter our arguments as explained in \cite{frauenfelderschlenk:equivariant:2016}.
\end{remark}

\begin{lemma}
	Let $(W,\lambda)$ be a Liouville domain and $\varphi \in \Aut(W,\lambda)$. Then there exists Rabinowitz--Floer data for $\varphi$.
	\label{lem:defining_Hamiltonian}
\end{lemma}

\begin{proof}
	The construction of the twisted defining Hamiltonian $H$ for $\Sigma$ is inspired by the proof of \cite[Proposition~4.1]{cieliebakfrauenfelderpaternain:mane:2010}. Fix $\delta > 0$ such that the exact symplectic embedding
	\begin{equation*}
		\psi \colon \del[1]{\intoc[0]{-\delta,0} \times \Sigma,e^r\lambda\vert_\Sigma} \hookrightarrow (W,\lambda)
	\end{equation*}
	\noindent defined by
	\begin{equation*}
		\psi(r,x) := \phi^X_r(x)
	\end{equation*}
	\noindent satisfies
	\begin{equation}
		\label{eq:assumption_delta}
		U_\delta := \psi(\intoc[0]{-\delta,0} \times \Sigma) \cap \supp f_\varphi = \emptyset.
	\end{equation}
	Indeed, that $\psi$ is an exact symplectic embedding follows from the computation 
	\begin{align*}
		\frac{d}{dr} \psi^*_r \lambda &= \frac{d}{dr}\del[1]{\phi_r^X}^* \lambda\\
		&= (\phi^X_r)^*L_X\lambda\\
		&= \del[1]{\phi_r^X}^*(di_X\lambda + i_Xd\lambda)\\
		&= \del[1]{\phi_r^X}^*(di_Xi_Xd\lambda + \lambda)\\
		&= \del[1]{\phi_r^X}^* \lambda\\
		&= \psi^*_r \lambda
	\end{align*}
	\noindent implying
	\begin{equation*}
		\psi^*_r\lambda = e^r\lambda\vert_\Sigma \qquad \forall r \in \intoc[0]{-\delta,0},
	\end{equation*}
	\noindent by $\psi_0 = \iota_\Sigma$, where $\iota_\Sigma \colon \Sigma \hookrightarrow W$ denotes the inclusion. Note that $\psi^*X = \partial_r$. We claim  
	\begin{equation}
		\label{eq:phi-invariance}
		\varphi(\psi(r,x)) = \psi(r,\varphi(x)) \qquad \forall (r,x) \in \intoc[0]{-\delta,0} 	\times \Sigma,
	\end{equation}
	\noindent that is, $\varphi$ and $\psi$ commute. Note that \eqref{eq:phi-invariance} makes sense because $\varphi(\Sigma) = \Sigma$ by assumption. Indeed, \eqref{eq:phi-invariance} follows from the uniqueness of integral curves and the computation
	\begin{align*}
		\frac{d}{dr}\varphi(\psi(r,x)) &= \frac{d}{dr}\varphi(\phi_r^X(x))\\
		&= D\varphi(X(\phi_r^X(x)))\\
		&= (D\varphi \circ X\vert_{U_\delta} \circ \varphi^{-1} \circ \varphi)(\phi_r^X(x))\\
		&= (\varphi_*X\vert_{\varphi(U_\delta)} \circ \varphi)(\phi_r^X(x))\\
		&= (X\vert_{\varphi(U_\delta)} \circ \varphi)(\phi_r^X(x))\\
		&= X(\varphi\del[1]{\psi(r,x)})
	\end{align*}
	\noindent where we used the $\varphi$-invariance of the Liouville vector field on $U_\delta$, that is,
	\begin{equation*}
		 \varphi_*X\vert_{\varphi(U_\delta)} = X\vert_{\varphi(U_\delta)},
	\end{equation*}
	\noindent which in turn follows from
	\allowdisplaybreaks
	\begin{align*}
		i_{\varphi_*X}d\lambda &= d\lambda(\varphi_*X,\cdot)\\
		&= d\lambda(D\varphi \circ X \circ \varphi^{-1},\cdot)\\
		&= d\lambda\del[1]{D\varphi \circ X \circ \varphi^{-1}, D\varphi \circ D\varphi^{-1}\cdot}\\
		&= \varphi^*d\lambda(X \circ \varphi^{-1}, D\varphi^{-1} \cdot)\\
		&= d\varphi^*\lambda(X \circ \varphi^{-1}, D\varphi^{-1} \cdot)\\
		&= d\lambda(X \circ \varphi^{-1}, D\varphi^{-1} \cdot)\\
		&= \varphi_*(i_Xd\lambda)\\
		&= \varphi_*\lambda\\
		&= \lambda - d(f_\varphi \circ \varphi^{-1})
	\end{align*}
	\noindent and assumption \eqref{eq:assumption_delta}.

\begin{figure}[h!tb]
	\centering
	\includegraphics[width=.56\textwidth]{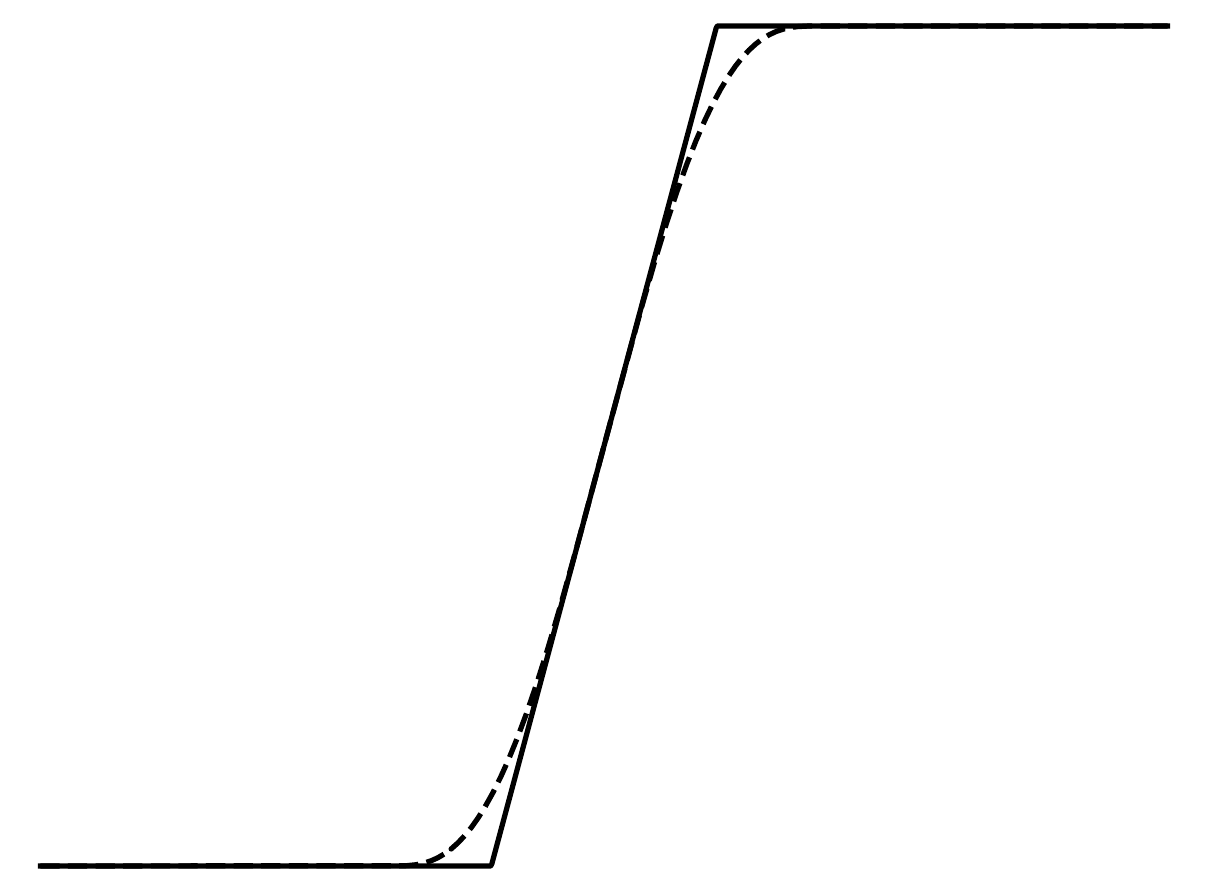}
	\caption{Mollification of the piecewise linear function $h$.}
	\label{fig:h}
\end{figure}

	Next we construct the defining Hamiltonian $H \in C^\infty(M)$. Let $h \in C^\infty(\mathbb{R})$ be a sufficiently small mollification of the piecewise linear function
	\begin{equation*}
		h(r) := \begin{cases}
			r & r \in \intcc[1]{-\frac{\delta}{2},\frac{\delta}{2}},\\
			\frac{\delta}{2} & r \in \intco[1]{\frac{\delta}{2},+\infty},\\
			-\frac{\delta}{2} & r \in \intoc[1]{-\infty,-\frac{\delta}{2}},
		\end{cases}
	\end{equation*}
	\noindent as in Figure \ref{fig:h}. 

	Define $H \in C^\infty(M)$ by
	\begin{equation}
		\label{eq:modified_Hamiltonian}
		H(p) := \begin{cases}
			h(r) & p = \psi(r,x) \in U_\delta,\\
			h(r) & p = (r,x) \in \intco[0]{0,+\infty} \times \Sigma,\\
			-\frac{\delta}{2} & p \in W \setminus U_\delta.
		\end{cases}
	\end{equation}
	Then $H$ is a defining Hamiltonian for $\Sigma$ and $dH$ is compactly supported by construction. Moreover, $H$ is $\varphi$-invariant by \eqref{eq:phi-invariance}. Finally, $X_H\vert_\Sigma = R$ follows from the observation $X_H = h'(r)e^{-r}R$. Indeed, on $U_\delta$ we compute
	\begin{align*}
		i_{h'(t)e^{-r}R}\psi^*d\lambda &= i_{h'(r)e^{-r}R} d(e^r\lambda\vert_\Sigma)\\
		&= i_{h'(r)e^{-r}R}(e^rdr \wedge \lambda\vert_\Sigma + e^r d\lambda\vert_\Sigma)\\
		&= -h'(r)dr\\
		&= -dH.
	\end{align*}
	
	Next we construct the family $J = (J_t)_{t \in \mathbb{R}}$ of $d\lambda$-compatible almost complex structures on $W$. Fix a $d\lambda\vert_\Sigma$-compatible almost complex structure $J$ on the contact distribution $\ker \lambda\vert_\Sigma$ and choose a path $(J^\Sigma_t)_{t \in I} \subseteq \mathcal{J}(\ker \lambda\vert_\Sigma,d\lambda\vert_\Sigma)$ from $J$ to $\varphi_*J$. Extend this smooth family to $(J_t^\Sigma)_{t \in \mathbb{R}}$ satisfying $\varphi^* J_{t + 1}^\Sigma = J_t^\Sigma$ for all $t \in \mathbb{R}$. Finally, extend this family to $(\intoo[0]{-\delta,+\infty} \times \Sigma, d(e^r\lambda\vert_\Sigma))$ by
\begin{equation}
	\label{eq:SFT-like}
	J_t^\Sigma\vert_{(a,x)}(b,v) := \del[1]{\lambda_x(v),  J_t^\Sigma\vert_x(\pi(v)) - bR(x)},
\end{equation}
\noindent where
\begin{equation*}
	\pi \colon \ker \lambda\vert_\Sigma \oplus \langle R \rangle \to \ker \lambda\vert_\Sigma
\end{equation*}
\noindent denotes the projection. Choose a smooth family $\del[1]{J^{W \setminus \Sigma}_t}_{t \in \mathbb{R}}$ on $W \setminus \Sigma$ twisted by $\varphi$, and let $\cbr[1]{\beta^\Sigma,\beta^{W \setminus \Sigma}}$ be a partition of unity subordinate to $\cbr[0]{U_\delta,W \setminus \Sigma}$. Define a smooth family $(m_t)_{t \in \mathbb{R}}$ of Riemannian metrics on $W$ by
\begin{equation*}
	m_t := \beta^\Sigma m_{\psi_*J^\Sigma_t} + \beta^{W \setminus \Sigma} m_{J^{W \setminus \Sigma}_t}
\end{equation*}
\noindent and let $(J_t)_{t \in \mathbb{R}}$ be the corresponding family of $d\lambda$-compatible almost complex structures on $W$.
\end{proof}

\begin{definition}[\bld{$L^2$}-Metric]
	Let $(H,J)$ be Rabinowitz--Floer data for a Liouville automorphism $\varphi \in \Aut(W,\lambda)$. Define an \bld{$L^2$-metric} on $\mathscr{L}_\varphi M \times \mathbb{R}$
	\begin{equation}
        \label{eq:L^2-metric}
        \langle (X,\eta),(Y,\sigma)\rangle_J := \int_0^1 d\lambda(J_t X(t),Y(t))dt + \eta\sigma
	\end{equation}
	\noindent for all $(X,\eta),(Y,\sigma) \in T_\gamma\mathscr{L}_\varphi M \times \mathbb{R}$ and $\gamma \in \mathscr{L}_\varphi M$.
\end{definition}

With respect to the $L^2$-metric \eqref{eq:L^2-metric}, the gradient of the twisted Rabinowitz action functional $\grad_J\mathscr{A}^H_\varphi \in \mathfrak{X}(\mathscr{L}_\varphi M \times \mathbb{R})$ is given by
\begin{equation*}
	\grad_J \mathscr{A}^H_\varphi\vert_{(\gamma,\tau)}(t) = \begin{pmatrix}
		J_t(\dot{\gamma}(t) - \tau X_H(\gamma(t)))\\
		\displaystyle-\int_0^1 H \circ \gamma
	\end{pmatrix} \qquad \forall t \in \mathbb{R}.
\end{equation*}

\begin{lemma}[Fundamental Lemma]
	Let $(H,J)$ be Rabinowitz--Floer data for a Liouville automorphism $\varphi \in \Aut(W,\lambda)$. Then there exists a constant $C = C(\lambda,H,J) > 0$ such that
	\begin{equation*}
		\norm[1]{\grad_J \mathscr{A}^H_\varphi\vert_{(\gamma,\tau)}}_J < \frac{1}{C} \quad \Rightarrow \quad \abs[0]{\tau} \leq C(\abs[0]{\mathscr{A}^H_\varphi(\gamma,\tau)} + 1)
	\end{equation*}
	\noindent for all $(\gamma,\tau) \in \mathscr{L}_\varphi M \times \mathbb{R}$.
	\label{lem:fundamental_lemma}
\end{lemma}

\begin{proof}
	The proof \cite[Proposition~3.2]{cieliebakfrauenfelder:rfh:2009} goes through with minor modifications as $\norm[0]{f_\varphi}_\infty < +\infty$ by assumption.
\end{proof}

\begin{definition}[Twisted Negative Gradient Flow Line]
	Let $(H,J)$ be Rabinowitz--Floer data for a Liouville automorphism $\varphi \in \Aut(W,\lambda)$. A \bld{twisted negative gradient flow line} is a tuple $(u,\tau) \in C^\infty(\mathbb{R}, \mathscr{L}_\varphi M \times \mathbb{R})$ such that
	\begin{equation*}
		\partial_s(u,\tau) = -\grad_J \mathscr{A}^H_\varphi\vert_{(u(s),\tau(s))} \qquad \forall s \in \mathbb{R}.
	\end{equation*}
\end{definition}

\begin{definition}[Energy]
	Let $(H,J)$ be Rabinowitz--Floer data for a Liouville automorphism $\varphi \in \Aut(W,\lambda)$. The \bld{energy of a twisted negative gradient flow line $(u,\tau)$} is defined by
	\begin{equation*}
		E_J(u,\tau) := \int_{-\infty}^{+\infty} \norm[0]{\partial_s(u,\tau)}^2_J ds = \int_{-\infty}^{+\infty} \norm[1]{\grad_J \mathscr{A}^H_\varphi\vert_{(u(s),\tau(s))}}^2_J ds.
	\end{equation*}
\end{definition}

\begin{theorem}[Compactness]
	Let $(H,J)$ be Rabinowitz--Floer data for a Liouville automorphism $\varphi \in \Aut(W,\lambda)$. Suppose $(u_\mu,\tau_\mu)$ is a sequence of negative gradient flow lines of the twisted Rabinowitz action functional $\mathscr{A}^H_\varphi$ such that there exist $a,b \in \mathbb{R}$ with
	\begin{equation*}
		a \leq \mathscr{A}^H_\varphi\del[1]{u_\mu(s),\tau_\mu(s)} \leq b \qquad \forall \mu \in \mathbb{N}, s \in \mathbb{R}.
	\end{equation*}
	Then for every reparametrisation sequence $(s_\mu) \subseteq \mathbb{R}$ there exists a subsequence $\mu_\nu$ of $\mu$ and a negative gradient flow line $(u_\infty,\tau_\infty)$ of $\mathscr{A}^H_\varphi$ such that 
	\begin{equation*}
		\del[1]{u_{\mu_\nu}(\cdot + s_{\mu_\nu}),\tau_{\mu_\nu}(\cdot + s_{\mu_\nu})} \to (u_\infty,\tau_\infty) \qquad \text{as } \nu \to \infty
	\end{equation*}
	\noindent in the $C^\infty_{\loc}(\mathbb{R}, \mathscr{L}_\varphi M \times \mathbb{R})$-topology.
	\label{thm:compactness}
\end{theorem}

\begin{proof}
	The proof \cite[p.~268]{cieliebakfrauenfelder:rfh:2009} goes through without any changes as we have a twisted version of the Fundamental Lemma. However, for convenience, we reproduce the main arguments here. We need to establish
	\begin{itemize}
		\item a uniform $L^\infty$-bound on $u_\mu$,
		\item a uniform $L^\infty$-bound on $\tau_\mu$,
		\item a uniform $L^\infty$-bound on the derivatives of $u_\mu$.
	\end{itemize}
	Indeed, by elliptic bootstrapping \cite[Theorem~B.4.1]{mcduffsalamon:J-holomorphic_curves:2012} the negative gradient flow equation we will obtain $C^\infty_{\loc}$-convergence by \cite[Theorem~B.4.2]{mcduffsalamon:J-holomorphic_curves:2012}.

	To obtain a uniform $L^\infty$-bound on the sequence of twisted negative gradient flow lines $u_\mu$, observe that by definition of Rabinowitz--Floer data for $\varphi$, there exists $r \in \intoo[0]{0,+\infty}$ with
	\begin{equation*}
		\supp X_H \cap \intco[0]{r,+\infty} \times \Sigma = \emptyset
	\end{equation*}
	\noindent and $J_t$ is adapted to the boundary of $W \cup_\Sigma \intcc[0]{0,r} \times \Sigma$ for all $t \in \mathbb{R}$. Consequently, \cite[Corollary~9.2.11]{mcduffsalamon:J-holomorphic_curves:2012} implies that every $u_\mu$ remains inside the compact set $W \cup_\Sigma \intcc[0]{0,r} \times \Sigma$ as the asymptotics belong to $W \cup_\Sigma \intco[0]{0,r} \times \Sigma$ for all $\mu \in \mathbb{N}$. Indeed, this follows from
	\allowdisplaybreaks
	\begin{align*}
		E_J(u_\mu,\tau_\mu) &= \int_{-\infty}^{+\infty} \norm[0]{\partial_s(u_\mu,\tau_\mu)}^2_J ds\\
		&= \int_{-\infty}^{+\infty} \langle \partial_s (u_\mu,\tau_\mu),\partial_s (u_\mu,\tau_\mu)\rangle_Jds\\
		&= -\int_{-\infty}^{+\infty} \langle\grad_J\mathscr{A}^H_\varphi\vert_{(u_\mu(s),\tau_\mu(s))},\partial_s (u_\mu,\tau_\mu)\rangle_Jds\\
		&= -\int_{-\infty}^{+\infty} d\mathscr{A}^H_\varphi(\partial_s (u_\mu,\tau_\mu))ds\\
		&= -\int_{-\infty}^{+\infty} \partial_s \mathscr{A}^H_\varphi(u_\mu,\tau_\mu)ds\\
		&= \lim_{s \to -\infty}\mathscr{A}^H_\varphi(u_\mu(s),\tau_\mu(s)) - \lim_{s \to +\infty} \mathscr{A}^H_\varphi(u_\mu(s),\tau_\mu(s))\\
		&\leq b - a,
	\end{align*}
	\noindent as this implies
	\begin{equation*}
		\lim_{s \to \pm \infty} \norm[0]{\partial_s(u_\mu,\tau_\mu)}_J = \lim_{s \to \pm \infty}\norm[1]{\grad_J \mathscr{A}^H_\varphi\vert_{(u_\mu(s),\tau_\mu(s))}}_J = 0
	\end{equation*}
	\noindent by the negative gradient flow equation. 

	The uniform $L^\infty$-bound on the Lagrange multipliers $\tau_\mu$ follows from the Fundamental Lemma \ref{lem:fundamental_lemma} by arguing as in \cite[Corollary~3.5]{cieliebakfrauenfelder:rfh:2009}. 

	Lastly, the uniform $L^\infty$-bound on the derivatives of $u_\mu$ follows from standard bubbling-off analysis. Indeed, if the derivatives of $u_\mu$ are unbounded, then there exists a nonconstant pseudoholomorphic sphere as in \cite[Section~4.2]{mcduffsalamon:J-holomorphic_curves:2012}. This is impossible as $M$ is an exact symplectic manifold and thus in particular symplectically aspherical.
\end{proof}

\section{Definition of Twisted Rabinowitz--Floer Homology} 
\label{sec:definition}
In this section we define twisted Rabinowitz--Floer homology for the setting introduced in the last section. Note that here we make implicit use of the requirement that a Liouville automorphism has finite order near the boundary. This is crucial because then the arguments go through as in the case of loops by Remark \ref{rem:finite_order}. 

\begin{definition}[Transverse Conley--Zehnder Index]
	\label{def:CZ-index}
	Let $(W^{2n},\lambda)$ be a Liouville domain with boundary $\Sigma$. Let $(\gamma_0,\tau_0), (\gamma_1,\tau_1) \in \mathscr{P}_\varphi(\Sigma,\lambda\vert_\Sigma)$ for some $\varphi \in \Aut(W,\lambda)$ such that there exists a path $\gamma_\sigma$ in $\mathscr{L}_\varphi \Sigma$ from $\gamma_0$ to $\gamma_1$. Define the \bld{transverse Conley--Zehnder index} by
	\begin{equation*}
		\mu((\gamma_0,\tau_0),(\gamma_1,\tau_1)) := \mu_{\CZ}(\Psi^1) - \mu_{\CZ}(\Psi^0) \in \mathbb{Z},
	\end{equation*}
	\noindent with
	\begin{align*}
		\Psi^0 \colon I \to \Sp(n - 1), \qquad &\Psi^0_t := \Phi_{t,0}^{-1} \circ D^\xi \theta^R_{\tau_0t} \circ \Phi_{0,0},\\
		\Psi^1 \colon I \to \Sp(n - 1), \qquad &\Psi^1_t := \Phi_{t,1}^{-1} \circ D^\xi \theta^R_{\tau_1t} \circ \Phi_{0,1},
	\end{align*}
	\noindent where $\Phi_{t,\sigma} \colon \mathbb{R}^{2n - 2} \to \xi_{\gamma_\sigma(t)}$ is a symplectic trivialisation of $F^*\xi$, $\xi := \ker \lambda\vert_\Sigma$ with $F \in C^\infty(\mathbb{R} \times I,M)$ being defined by $F(t,\sigma) := \gamma_\sigma(t)$, satisfying
	\begin{equation*}
		\Phi_{t + 1,\sigma} = D\varphi \circ \Phi_{t,\sigma} \qquad \forall (t,\sigma) \in \mathbb{R} \times I.
	\end{equation*}	
\end{definition}
 
\begin{remark}
	The transverse Conley--Zehnder index, or more precisely, the twisted relative transverse Conley--Zehnder index, does not depend on the choice of trivialisation. Denote by 
	\begin{equation*}
		\Sigma_\varphi := \frac{\Sigma \times \mathbb{R}}{(\varphi(x),t + 1){\sim}(x,t)}
	\end{equation*}
	\noindent the mapping torus of $\varphi$ giving rise to the fibration
	\begin{equation*}
		\pi_\varphi \colon \Sigma_\varphi \to \mathbb{S}^1, \qquad \pi_\varphi([x,t]) := [t].
	\end{equation*}
	The vertical bundle $\ker D^\xi \pi_\varphi \to \Sigma_\varphi$ is a symplectic vector bundle. One can show, that if $c_1(\ker D^\xi\pi_\varphi) = 0$, then the transverse Conley--Zehnder index is independent of the choice of path in $\mathscr{L}_\varphi \Sigma$ from $\gamma_0$ to $\gamma_1$.
	\label{rem:grading}
\end{remark}

Let $(H,J)$ be Rabinowitz--Floer data for $\varphi \in \Aut(W,\lambda)$. Set
\begin{equation*}
	\Sigma := \partial W \qquad \text{and} \qquad M := W \cup_\Sigma \intco[0]{0, +\infty} \times \Sigma.
\end{equation*}
Fix $(\eta,\tau_\eta) \in \mathscr{P}_\varphi(\Sigma,\lambda\vert_\Sigma)$ and denote by $[\eta]$ the corresponding class in $\pi_0 \mathscr{L}_\varphi \Sigma$. Assume that the twisted Rabinowitz action functional $\mathscr{A}^H_\varphi$ is Morse--Bott, that is, $\Crit \mathscr{A}^H_\varphi \subseteq \Sigma \times \mathbb{R}$ is a properly embedded submanifold by Proposition \ref{prop:kernel_hessian_contact}, and fix a Morse function $h \in C^\infty(\Crit \mathscr{A}^H_\varphi)$. Define the twisted Rabinowitz--Floer chain group $\RFC^\varphi(\Sigma,M)$ to be the $\mathbb{Z}_2$-vector space consisting of all formal linear combinations
\begin{equation*}
	\zeta = \sum_{\substack{(\gamma,\tau) \in \Crit(h)\\{[\gamma] = [\eta]}}} \zeta_{(\gamma,\tau)} (\gamma,\tau)
\end{equation*}
\noindent satisfying the Novikov finiteness condition
\begin{equation*}
	\#\cbr[0]{(\gamma,\tau) \in \Crit(h) : \zeta_{(\gamma,\tau)} \neq 0, \mathscr{A}^H_\varphi(\gamma,\tau) \geq \kappa} < \infty \qquad \forall \kappa \in \mathbb{R}.
\end{equation*}
Define a boundary operator
\begin{equation*}
	\partial \colon \RFC^\varphi(\Sigma,M) \to \RFC^\varphi(\Sigma,M) 
\end{equation*}
\noindent by
\begin{equation*}
	\partial (\gamma^-,\tau^-) := \sum_{\substack{(\gamma^+,\tau^+) \in \Crit(h)\\{[\gamma^+] = [\gamma^-]}}} n_\varphi(\gamma^\pm,\tau^\pm)(\gamma^+,\tau^+),
\end{equation*}
\noindent where 
\begin{equation*}
	n_\varphi(\gamma^\pm,\tau^\pm) := \#_2\mathcal{M}^1_\varphi(\gamma^\pm,\tau^\pm)/\mathbb{R} \in \mathbb{Z}_2,
\end{equation*}
\noindent with $\mathcal{M}^1_\varphi(\gamma^\pm,\tau^\pm)$ being the one-dimensional component of the moduli space of all twisted negative gradient flow lines with cascades from $(\gamma^-,\tau^-)$ to $(\gamma^+,\tau^+)$. This is well-defined by the Compactness Theorem \ref{thm:compactness}. Define the \bld{twisted Rabinowitz--Floer homology of $\Sigma$ and $\varphi$} by
\begin{equation*}
	\RFH^\varphi(\Sigma,M) := \frac{\ker \partial}{\im \partial}.
\end{equation*}
 
\begin{proposition}
	\label{prop:fixed_points}
	Let $(W,\lambda)$ be a Liouville domain with simply connected boundary $\Sigma$ and $\varphi \in \Aut(W,\lambda)$. If there do not exist any nonconstant twisted periodic Reeb orbits on $\Sigma$, then
	\begin{equation*}
		\RFH^\varphi_\ast(\Sigma,M) \cong \operatorname{H}_\ast(\Fix(\varphi\vert_\Sigma);\mathbb{Z}_2).
	\end{equation*}	
\end{proposition}

\begin{proof}
	If there do not exist any nonconstant twisted periodic Reeb orbits,
	\begin{equation*}
		\Crit \mathscr{A}^H_\varphi = \cbr[0]{(c_x,0) : x \in \Fix(\varphi\vert_\Sigma)} \cong \Fix(\varphi\vert_\Sigma)
	\end{equation*}
	\noindent for any $H \in \mathscr{F}_\varphi(\Sigma)$. Since $\Fix(\varphi\vert_\Sigma)$ is a properly embedded submanifold of $\Sigma$ by \cite[Problem~8-32]{lee:dg:2018} or \cite[Lemma~5.5.7]{mcduffsalamon:st:2017}, $\mathscr{A}^H_\varphi$ is Morse--Bott. Let $x,y \in \Fix(\varphi\vert_\Sigma)$. As $\Sigma$ is simply connected by assumption, there exists a path $\gamma$ from $x$ to $y$ in $\Sigma$ and a homotopy from $\gamma$ to $\varphi \circ \gamma$ with fixed endpoints. Extend this homotopy to a path in $\mathscr{L}_\varphi \Sigma$ from $c_x$ to $c_y$. Choose a Morse function $h$ on $\Fix(\varphi\vert_\Sigma)$ and any critical point $c_x \in \Fix(\varphi\vert_\Sigma)$. Then a $\mathbb{Z}$-grading of $\RFC^\varphi(\Sigma,M)$ is given by the index 
	\begin{equation*}
		\mu((c_y,0),(c_x,0)) + \ind_h(c_y) = \ind_h(c_y) \qquad \forall c_y \in \Crit(h),
	\end{equation*}
	\noindent and consequently,
	\begin{equation*}
		\RFH^\varphi_\ast(\Sigma,M) = \HM_\ast(\Fix(\varphi\vert_\Sigma);\mathbb{Z}_2) \cong \operatorname{H}_\ast(\Fix(\varphi\vert_\Sigma);\mathbb{Z}_2)
	\end{equation*}
	\noindent as there are only twisted negative gradient flow lines with zero cascades, that is, ordinary Morse gradient flow lines of $h$. Indeed, suppose that there is a nonconstant twisted negative gradient flow line $(u,\tau)$ of $\mathscr{A}^H_\varphi$ such that
	\begin{equation*}
		\lim_{s \to \pm \infty} (u(s),\tau(s)) = (\gamma^\pm,\tau^\pm) \in \Crit \mathscr{A}^H_\varphi.
	\end{equation*}
	Using the twisted negative gradient flow equation we estimate
	\allowdisplaybreaks
	\begin{align*}
		\tau_- - \tau_+ &= \mathscr{A}^H_\varphi(\gamma_-,\tau_-) - \mathscr{A}^H_\varphi(\gamma_+,\tau_+)\\
		&= \lim_{s \to -\infty} \mathscr{A}^H_\varphi(u(s),\tau(s)) - \lim_{s \to +\infty} \mathscr{A}^H_\varphi(u(s),\tau(s))\\
		&= -\int_{-\infty}^{+\infty} \partial_s \mathscr{A}^H_\varphi(u,\tau)ds\\
		&= -\int_{-\infty}^{+\infty} d\mathscr{A}^H_\varphi(\partial_s(u,\tau)) ds\\
		&= \int_{-\infty}^{+\infty} d\mathscr{A}^H_\varphi(\grad_J \mathscr{A}^H_\varphi\vert_{(u(s),\tau(s))})ds\\
		&= \int_{-\infty}^{+\infty} \norm[1]{\grad_J \mathscr{A}^H_\varphi\vert_{(u(s),\tau(s))}}^2_J ds\\
		&> 0.
	\end{align*}
	Hence $\tau_+ < \tau_-$, contradicting $\tau^\pm = 0$. 
\end{proof}

\section{Invariance of Twisted Rabinowitz--Floer Homology Under Twisted Homotopies of Liouville Domains}
\label{sec:invariance}

\begin{definition}[Twisted Homotopy of Liouville Domains]
	Given the completion $(M,\lambda)$ of a Liouville domain $(W_0,\lambda)$ and $\varphi \in \Aut(W,\lambda)$, a \bld{twisted homotopy of Liouville domains in $M$} is a time-dependent Hamiltonian function $H \in C^\infty(M \times I)$ such that
	\begin{enumerate}[label=\textup{(\roman*)}]
		\item $W_\sigma := H^{-1}_\sigma(\intoc[0]{-\infty,0})$ is a Liouville domain with symplectic form $d\lambda\vert_{W_\sigma}$ and boundary $\Sigma_\sigma := H^{-1}_\sigma(0)$ for all $\sigma \in I$,
		\item $H_\sigma \in \mathscr{F}_\varphi(\Sigma_\sigma)$ for all $\sigma \in I$,
		\item $\Sigma_\sigma \cap \supp f_\varphi = \emptyset$ for all $\sigma \in I$.
	\end{enumerate}
\end{definition}

Twisted Rabinowitz--Floer homology is stable under twisted homotopies of Liouville domains. This property is crucial for proving Theorem \ref{thm:noncontractible}.

\begin{theorem}[Invariance of Twisted Rabinowitz--Floer Homology]
	\label{thm:invariance}
	If $(H_\sigma)_{\sigma \in I}$ is a twisted homotopy of Liouville domains such that both $\mathscr{A}^{H_0}_\varphi$ and $\mathscr{A}^{H_1}_\varphi$ are Morse--Bott, then there is a canonical isomorphism
	\begin{equation*}
		\RFH^\varphi(\Sigma_0,M) \cong \RFH^\varphi(\Sigma_1,M).
	\end{equation*}
\end{theorem}

\begin{proof}
	The proof follows from the same adiabatic argument as in \cite[p.~275--277]{cieliebakfrauenfelder:rfh:2009}. Crucial is that \cite[Theorem~3.6]{cieliebakfrauenfelder:rfh:2009} remains true in our setting as well as the genericity of the Morse--Bott condition. Indeed, if $(M,\lambda)$ is an exact symplectic manifold and $\varphi \in \Diff(M)$ is of finite order such that $\varphi^*\lambda = \lambda$, then we have the following generalisation of \cite[Theorem~B.1]{cieliebakfrauenfelder:rfh:2009}. Adapting the proof accordingly, one can show that there exists a subset
	\begin{equation*}
		\mathscr{U} \subseteq \{H \in C^\infty_\varphi(M) : \supp dH \text{ compact}\},
	\end{equation*}
	\noindent of the second category such that for every $H \in \mathscr{U}$, $\mathscr{A}^H_\varphi$ is Morse--Bott with critical manifold being $\Fix(\varphi\vert_{H^{-1}(0)})$ together with a disjoint union of circles. Again, this works only since the contact condition is an open condition.	
\end{proof}

\begin{remark}
	Invariance of twisted Rabinowitz--Floer homology allows us to define twisted Rabinowitz--Floer homology also in the case where $\mathscr{A}^H_\varphi$ is not necessarily Morse--Bott. Indeed, as the proof of Theorem \ref{thm:invariance} shows, we can perturb the hypersurface $\Sigma$ slightly to make it Morse--Bott. Thus we can define the twisted Rabinowitz--Floer homology of such a hypersurface to be the twisted Rabinowitz--Floer homology of any Morse--Bott perturbation. This is well-defined by Theorem \ref{thm:invariance}.
\end{remark}

\begin{corollary}[Independence]
	Let $\varphi \in \Aut(W,\lambda)$ and $H_0, H_1 \in \mathscr{F}_\varphi(\Sigma)$ be such that either $\mathscr{A}^{H_0}_\varphi$ or $\mathscr{A}^{H_1}_\varphi$ is Morse--Bott. Then the definition of twisted Rabinowitz--Floer homology $\RFH^\varphi(\Sigma,M)$ is independent of the choice of twisted defining Hamiltonian function for $\Sigma$.
\end{corollary}

\begin{proof}
	Note that $\mathscr{F}_\varphi(\Sigma)$ is a convex space. Indeed, set 
	\begin{equation*}
		H_\sigma := (1 - \sigma)H_0 + \sigma H_1 \qquad \sigma \in I.
	\end{equation*}
	Then $\varphi^*H_\sigma = H_\sigma$, $dH_\sigma$ has compact support and $X_{H_\sigma}\vert_\Sigma = R$ for all $\sigma \in I$. Moreover, for the Liouville vector field $X \in \mathfrak{X}(M)$ we compute
	\begin{equation*}
		\frac{d}{dt}\bigg\vert_{t = 0} H \circ \phi^X_t\vert_\Sigma = dH(X)\vert_\Sigma = d\lambda(X,X_H)\vert_\Sigma = \lambda(X_H)\vert_\Sigma = \lambda(R) = 1, 
	\end{equation*}
	\noindent for any $H \in \mathscr{F}_\varphi(\Sigma)$, and thus $H < 0$ on $\Int W$ and $H > 0$ on $M \setminus W$. Consequently, $H^{-1}_\sigma(0) = \Sigma$ and so $H_\sigma \in \mathscr{F}_\varphi(\Sigma)$ for all $\sigma \in I$. Hence $(H_\sigma)_{\sigma \in I}$ is a twisted homotopy of Liouville domains in $M$ and Theorem \ref{thm:invariance} implies the claim.
\end{proof}

\section{Twisted Leaf-Wise Intersection Points}

\begin{definition}[Twisted Leaf-Wise Intersection Point]
	Let $(M,\lambda)$ be the completion of a Liouville domain $(W,\lambda)$ and let $\varphi \in \Aut(W,\lambda)$ be a Liouville automorphism. A point $x \in \Sigma$ is a \bld{twisted leaf-wise intersection point} for a Hamiltonian symplectomorphism $\varphi_F \in \Ham(M,d\lambda)$, iff
	\begin{equation*}
		\varphi_F(x) \in L_{\varphi(x)} = \cbr[0]{\phi^R_t(\varphi(x)) : t \in \mathbb{R}}.
	\end{equation*}
\end{definition}

\begin{definition}[Twisted Moser Pair]
	Let $\varphi \in \Aut(W,\lambda)$. A \bld{twisted Moser pair} is defined to be a tuple $\mathfrak{M} := (\chi H,F)$, where 
	\begin{enumerate}[label=\textup{(\roman*)}]
		\item $H \in C^\infty_\varphi(M)$, $F \in C^\infty_\varphi(M \times \mathbb{R})$ and $\chi \in C^\infty(\mathbb{S}^1,I)$ such that $\int_0^1 \chi = 1$. Any time-dependent Hamiltonian function $\chi H$ is said to be \bld{weakly time-dependent}.
		\item $\supp \chi \subseteq \intoo[1]{0,\frac{1}{2}}$ and $F_t = 0$ for all $t \in \intcc[1]{0,\frac{1}{2}}$.
	\end{enumerate}
\end{definition}

\begin{lemma}	
	Let $\varphi \in \Aut(W,\lambda)$. For all $H \in \mathscr{F}_\varphi(\Sigma)$ and $\varphi_F \in \Ham(M,d\lambda)$ there exists a corresponding twisted Moser pair $\mathfrak{M}$ such that the flow of $\chi X_H$ is a time-reparametrisation of the flow of $X_H$.  
	\label{lem:twisted_Moser_pair}
\end{lemma}

\begin{proof}
	For constructing the Hamiltonian perturbation $\tilde{F}$, fix $\rho \in C^\infty(I,I)$ such that 
	\begin{equation*}
		\rho(t) = \begin{cases}
			0 & t \in \intcc[1]{0,\frac{1}{2}},\\
			1 & t \in \intcc[1]{\frac{2}{3},1}.
		\end{cases}
	\end{equation*}
	See Figure \ref{fig:rho}. Then define $\tilde{F} \in C^\infty_\varphi(M \times \mathbb{R})$ by 
	\begin{equation*}
		\tilde{F}(x,t) := \dot{\rho}(t - k)F\del[1]{\varphi^{-k}(x),\rho(t - k)} \qquad \forall t \in \intcc[0]{k,k + 1},
	\end{equation*}
	\noindent for $k \in \mathbb{Z}$. See Figure \ref{fig:rho_dot}. Then $\tilde{F}_t = 0$ for all $t \in \intcc[1]{0,\frac{1}{2}}$, and 
	\begin{equation*}
		\phi^{\tilde{F}}_t = \phi^F_{\rho(t)} \qquad \forall t \in I,
	\end{equation*}
	\noindent where $\phi$ denotes the smooth flow of the induced time-dependent Hamiltonian vector field. Indeed, we compute
	\begin{equation*}
		\frac{d}{dt}\phi^F_{\rho(t)} = \dot{\rho}(t)\frac{d}{d\rho}\phi^F_{\rho(t)} = \dot{\rho}(t)\del[1]{X_{F_{\rho(t)}} \circ \phi^F_{\rho(t)}} = X_{\tilde{F}_t} \circ \phi^F_{\rho(t)}.
	\end{equation*}
	In particular
	\begin{equation*}
		\varphi_{\tilde{F}} = \phi^{\tilde{F}}_1 = \phi^F_{\rho(1)} = \phi^F_1 = \varphi_F.
	\end{equation*}
	Finally, we have that
	\begin{equation*}
		\phi^{\chi X_H}_t = \phi^H_{\tau(t)} \qquad \text{with} \qquad \tau(t) := \int_0^t \chi.
	\end{equation*}
	Indeed, we compute
	\begin{equation*}
		\frac{d}{dt}\phi^H_ {\tau(t)} = \chi(t) \frac{d}{d\tau} \phi^H_{\tau(t)} = \chi(t)X_H \circ \phi^H_{\tau(t)},
	\end{equation*}
	\noindent and thus we conclude by the uniqueness of integral curves.
\end{proof}

\begin{figure}[h!tb]
    \centering
    \begin{subfigure}[b]{0.48\textwidth}
		\centering
        \includegraphics[width=\textwidth]{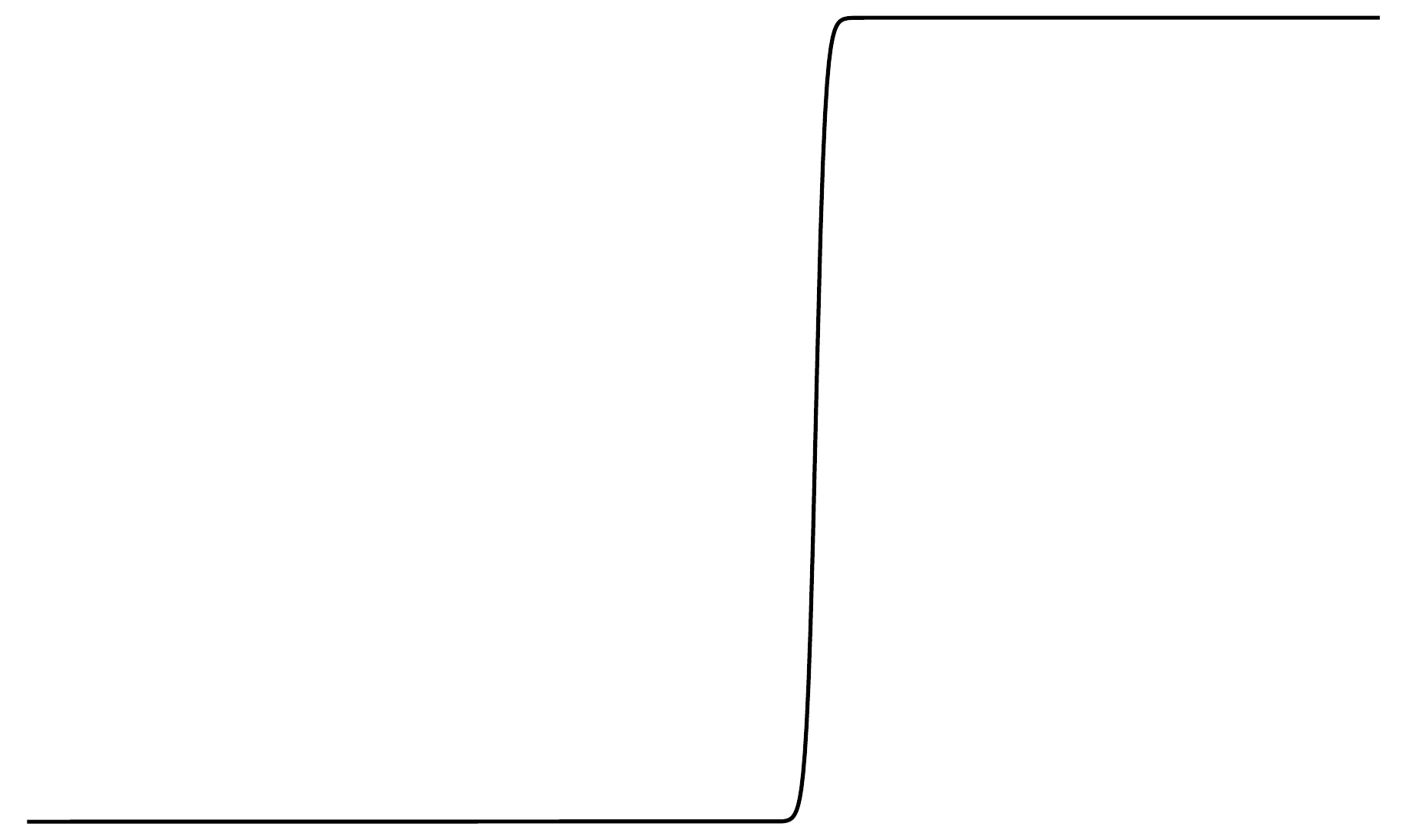}
        \caption{The smooth function $\rho$.}
        \label{fig:rho}
    \end{subfigure}
    \begin{subfigure}[b]{0.48\textwidth}
		\centering
        \includegraphics[width=\textwidth]{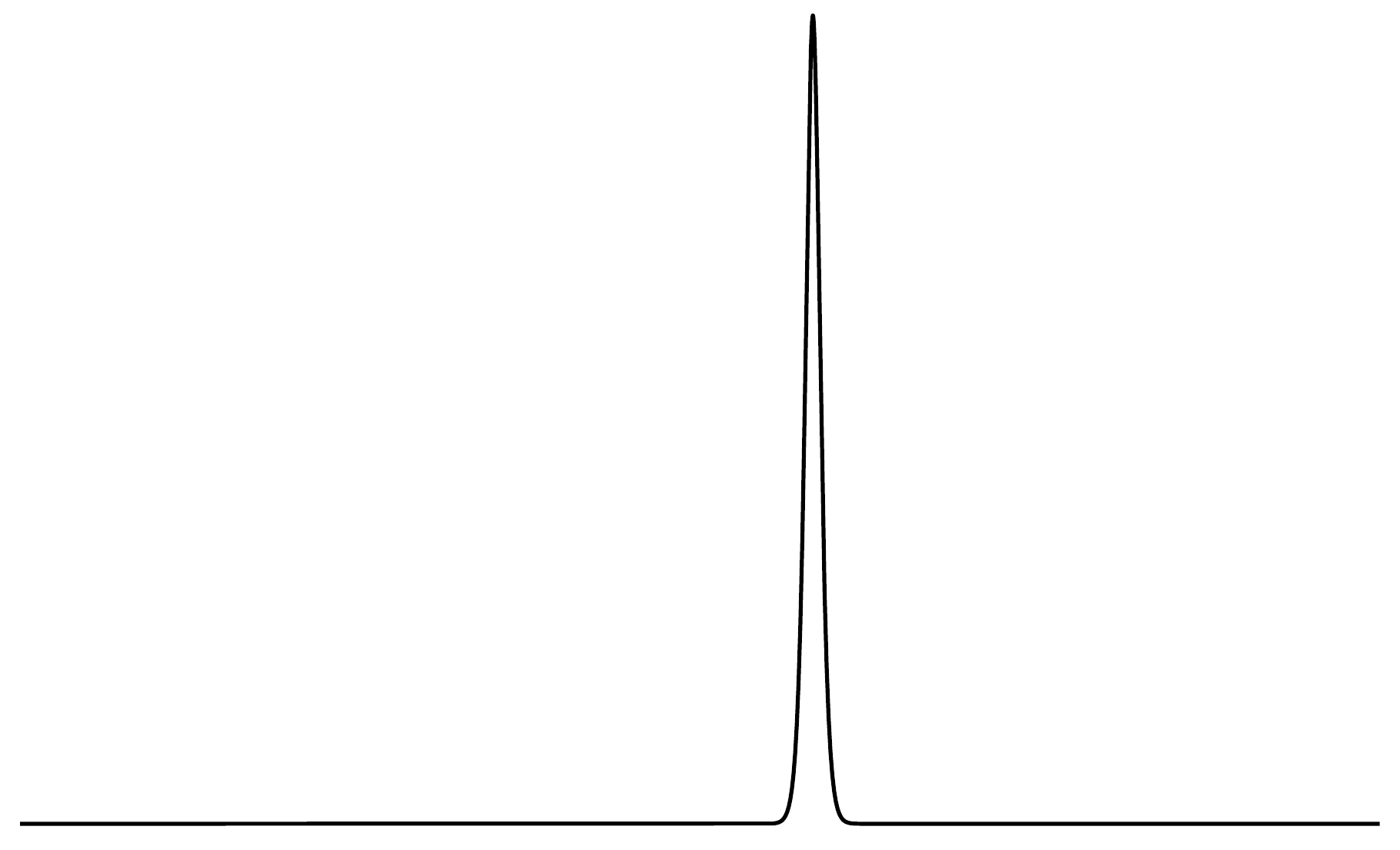}
		\caption{The derivative $\dot{\rho}$ of $\rho$.}
        \label{fig:rho_dot}
    \end{subfigure}
	\caption{}
\end{figure}

\begin{lemma}
	Let $\varphi \in \Aut(W,\lambda)$ and $\varphi_F \in \Ham(M,d\lambda)$ a Hamiltonian symplectomorphism. If $(\gamma,\tau) \in \Crit \mathscr{A}^{\mathfrak{M}}_\varphi$, then $x := \gamma\del[1]{\frac{1}{2}}$ is a twisted leaf-wise intersection point for $\varphi_F$.
	\label{lem:variational_characterisation_twisted_leaf-wise_intersections}
\end{lemma}

\begin{proof}
	Let $\mathfrak{M} = (\chi H,F)$ denote the twisted Moser pair from Lemma \ref{lem:twisted_Moser_pair}. Using Proposition \ref{prop:differential_perturbed_twisted_Rabinowitz_functional} we compute 
	\begin{align*}
		\frac{d}{dt}H(\gamma(t)) &= dH(\dot{\gamma}(t))\\
		&= dH(\tau X_{\chi(t)H}(\gamma(t)) + X_{F_t}(\gamma(t)))\\
		&= dH(\tau \chi(t) X_H(\gamma(t)))\\
		&= \tau\chi(t)dH(X_H(\gamma(t)))\\
		&= 0
	\end{align*}
	\noindent for all $t \in \intcc[1]{0,\frac{1}{2}}$. Thus $H \circ \gamma = c \in \mathbb{R}$ on $\intcc[1]{0,\frac{1}{2}}$ with 
	\begin{equation*}
		0 = \int_0^1 \chi H(\gamma) = \int_0^{\frac{1}{2}}\chi H(\gamma) = c \int_0^{\frac{1}{2}}\chi = c\int_0^1\chi = c. 
	\end{equation*}
	Consequently, $\gamma(0) \in L_x$ and $x \in \Sigma$. Moreover, also $\gamma(1) = \varphi(\gamma(0)) \in \Sigma$ by the $\varphi$-invariance of $H$. For $t \in \intcc[1]{\frac{1}{2},1}$, $\dot{\gamma} = X_{F_t}(\gamma)$ and so $\gamma(1) = \varphi_F(x) \in \Sigma$. We conclude 
	\begin{equation*}
		L_{\varphi(x)} = \{\phi^R_t(\varphi(x)) : t \in \mathbb{R}\} = \{\varphi(\phi^R_t(x)) : t \in \mathbb{R}\} = \varphi(L_x)
	\end{equation*}
	\noindent and so $\varphi_F(x) = \gamma(1) = \varphi(\gamma(0)) \in L_{\varphi(x)}$.
\end{proof}

\begin{theorem}
	Let $(W,\lambda)$ be a Liouville domain with displaceable boundary in the completion $(M,\lambda)$ and $\varphi \in \Aut(W,\lambda)$. Then $\RFH^\varphi(\Sigma,M) \cong 0$.
	\label{thm:displaceable}
\end{theorem}

\begin{proof}
	Suppose that $\Sigma = \partial W$ is displaceable in $M$ via $\varphi_F \in \Ham_c(M,d\lambda)$ and choose Rabinowitz--Floer data $(H,J)$ for $\varphi$. Denote by $\mathfrak{M} = (\chi H,F)$ the associated twisted Moser pair from Lemma \ref{lem:twisted_Moser_pair}. Then $\Crit \mathscr{A}^{\mathfrak{M}}_\varphi = \emptyset$. Indeed, if there exists $(\gamma,\tau) \in \Crit \mathscr{A}^{\mathfrak{M}}_\varphi$, then $\gamma\del[1]{\frac{1}{2}}$ is a twisted leaf-wise intersection point for $\varphi_F$ by Lemma \ref{lem:variational_characterisation_twisted_leaf-wise_intersections}. However, this is impossible as by displaceability we have that $\varphi_F(\Sigma) \cap \Sigma = \emptyset$. Consequently, the perturbed twisted Rabinowitz action functional $\mathscr{A}^{\mathfrak{M}}_\varphi$ is a Morse function. By adapting the Fundamental Lemma to the current setting as in \cite[Theorem~2.9]{albersfrauenfelder:rfh:2010}, the Floer homology $\HF(\mathscr{A}^{\mathfrak{M}}_\varphi)$ is well-defined. By making use of continuation homomorphisms we have that
	\begin{equation*}
		0 = \HF(\mathscr{A}^{\mathfrak{M}}_\varphi) \cong \HF(\mathscr{A}^{(\chi H,0)}_\varphi) \cong \RFH^\varphi(\Sigma,M),
	\end{equation*}
	\noindent where the last equation is the observation that twisted Rabinowitz--Floer homology in the autonomous case extends to the weakly time-dependent case without any issues. Crucial is, that the period--action equality (see Remark \ref{rem:period-action_equality})  is still valid. Indeed, we compute
	\begin{equation*}
		\mathscr{A}^{(\chi H,0)}_\varphi(\gamma,\tau) = \int_0^1 \gamma^*\lambda = \int_0^1 \lambda(\dot{\gamma}) = \tau \int_0^1 \chi \lambda(R(\gamma)) = \tau \int_0^1 \chi = \tau
	\end{equation*}
	\noindent for all $(\gamma,\tau) \in \Crit \mathscr{A}^{(\chi H,0)}_\varphi$.
\end{proof}

\section{Existence of Noncontractible Periodic Reeb Orbits}
We define an equivariant version of twisted Rabinowitz--Floer homology following \cite[p.~487]{albersfrauenfelder:eh:2012}. Denote by $(\mathbb{C}^n,\omega)$ the standard symplectic vector space with symplectic form
\begin{equation*}
	\omega := \sum_{j = 1}^n dy^j \wedge dx^j = \frac{i}{2}\sum_{j = 1}^n d\bar{z}^j \wedge dz^j,
\end{equation*}
\noindent and coordinates $z^j := x^j + iy^j$. Then $\omega = d\lambda$ for
\begin{equation}
	\lambda := \frac{1}{2}\sum_{j = 1}^n \del[1]{y^j dx^j - x^j dy^j} = \frac{i}{4} \sum_{j = 1}^n \del[1]{\bar{z}^j dz^j - z^j d\bar{z}^j}. 
	\label{eq:Liouville_form}
\end{equation}
Consider the free smooth discrete action on the odd-dimensional sphere 
	\begin{equation*}
		\mathbb{S}^{2n - 1} := \cbr[4]{(z^1,\dots,z^n) \in \mathbb{C}^n : \sum_{j = 1}^n\abs[0]{z^j}^2 = 1}
	\end{equation*}
	\noindent generated by the rotation
	\begin{equation*}
		\varphi \colon \mathbb{C}^n \to \mathbb{C}^n, \quad \varphi(z^1,\dots,z^n) := \del[1]{e^{2\pi i k_1/m}z^1,\dots,e^{2\pi i k_n/m}z^n}
	\end{equation*}
	\noindent for $m \geq 1$ and $k_1,\dots,k_n \in \mathbb{Z}$ coprime to $m$. Define a twisted defining Hamiltonian function $H \in \mathscr{F}_\varphi(\mathbb{S}^{2n - 1})$ by
\begin{equation*}
	H(z) := \frac{1}{2}\del[1]{\beta(\abs[0]{z}^2) - 1}
\end{equation*}
\noindent for some sufficiently small mollification of the piecewise linear function 
\begin{equation*}
	\beta(r) = \begin{cases}
		\frac{1}{2} & r \in \intoc[1]{-\infty,\frac{1}{2}},\\
		r & r \in \intcc[1]{\frac{1}{2},\frac{3}{2}},\\
		\frac{3}{2} & r \in \intco[1]{\frac{3}{2},+\infty}. 
	\end{cases}
\end{equation*}
Fix a $\varphi$-invariant $\omega$-compatible almost complex structure on $(\mathbb{C}^n,\omega)$. Then the rotation $\varphi$ induces a free $\mathbb{Z}_m$-action on $\Crit \mathscr{A}^H_\varphi$ and on the moduli space of twisted negative gradient flow lines with cascades of $\mathscr{A}^H_\varphi$. Therefore, we can define \bld{$\mathbb{Z}_m$-equivariant twisted Rabinowitz-Floer homology}
\begin{equation*}
	\overline{\RFH}^\varphi_k(\mathbb{S}^{2n - 1}/\mathbb{Z}_m) := \frac{\ker \bar{\partial}_k}{\im \bar{\partial}_{k + 1}} \qquad \forall k \in \mathbb{Z},
\end{equation*}
\noindent as the homology of the $\mathbb{Z}$-graded chain complex (see Remark \ref{rem:grading})
\begin{equation*}
	\bar{\partial}_k \colon \RFC^\varphi_k(\mathbb{S}^{2n - 1},\mathbb{C}^n)/\mathbb{Z}_m \to \RFC^\varphi_{k - 1}(\mathbb{S}^{2n - 1},\mathbb{C}^n)/\mathbb{Z}_m
\end{equation*}
\noindent given by
\begin{equation*}
	\bar{\partial}_k[(\gamma,\tau)] := [\partial_k(\gamma,\tau)] \qquad (\gamma,\tau) \in \Crit h,
\end{equation*}
\noindent for some $\varphi$-invariant Morse function $h$ on $\Crit \mathscr{A}^H_\varphi$.

\begin{theorem}
	Let $n \geq 2$. For $m \geq 1$ consider the rotation
	\begin{equation*}
		\varphi \colon \mathbb{C}^n \to \mathbb{C}^n, \quad \varphi(z^1,\dots,z^n) := \del[1]{e^{2\pi i k_1/m}z^1,\dots,e^{2\pi i k_n/m}z^n}
	\end{equation*}
	\noindent for $k_1,\dots,k_n \in \mathbb{Z}$ coprime to $m$. Then
	\begin{equation*}
		\overline{\RFH}^\varphi_k(\mathbb{S}^{2n - 1}/\mathbb{Z}_m) \cong \begin{cases}
			\mathbb{Z}_2 & m \text{ even},\\
			0 & m \text{ odd},
		\end{cases} \qquad \forall k \in \mathbb{Z},
	\end{equation*}
	If $m$ is even, then $\overline{\RFH}^\varphi_k(\mathbb{S}^{2n - 1}/\mathbb{Z}_m)$ is generated by a noncontractible periodic Reeb orbit in the lens space $\mathbb{S}^{2n - 1}/\mathbb{Z}_m$ for all $k \in \mathbb{Z}$. 
	\label{thm:equivariant_twisted_rfh}
\end{theorem}

\begin{proof}
	First we consider the special case
	\begin{equation*}
		\varphi \colon \mathbb{C}^n \to \mathbb{C}^n, \qquad \varphi(z) = e^{2\pi i/m}z.
	\end{equation*}
	The hypersurface $\mathbb{S}^{2n - 1} \subseteq (\mathbb{C}^n,d\lambda)$ is of restricted contact type with contact form $\lambda\vert_{\mathbb{S}^{2n - 1}}$ and associated Reeb vector field
\begin{equation*}
	R = 2\del[3]{y^j \frac{\partial}{\partial x^j} - x^j\frac{\partial}{\partial y^j}}\bigg\vert_{\mathbb{S}^{2n - 1}}= 2i\del[3]{\bar{z}\frac{\partial}{\partial \bar{z}} - z\frac{\partial}{\partial z}}\bigg\vert_{\mathbb{S}^{2n - 1}}.
\end{equation*}
Suppose $(\gamma,\tau) \in \Crit \mathscr{A}^H_\varphi$. If $\tau = 0$, then $\gamma$ is constant. This cannot happen as $\Fix(\varphi\vert_{\mathbb{S}^{2n - 1}}) = \emptyset$. So assume $\tau \neq 0$. Define a reparametrisation
	\begin{equation*}
		\gamma_\tau \colon \mathbb{R} \to \mathbb{S}^{2n - 1}, \qquad \gamma_\tau(t) := \gamma(t/\tau).
	\end{equation*}
	Then $\gamma_\tau$ is the unique integral curve of $R$ starting at $z := \gamma(0)$ and thus
	\begin{equation*}
		\gamma_\tau(t) = e^{-2it}z \qquad \forall t \in \mathbb{R}.
	\end{equation*}
	From $\gamma(t) = \gamma_\tau(\tau t)$ and the requirement
	\begin{equation*}
		e^{-2i\tau}z = \gamma(1) = \varphi(\gamma(0)) = \varphi(z) = e^{2\pi i/m}z,
	\end{equation*}
	\noindent we conclude $\tau \in \frac{\pi}{m}(m\mathbb{Z} - 1)$. Hence
	\begin{equation*}
		\Crit \mathscr{A}^H_\varphi = \cbr[1]{\del[1]{\phi^{\tau_k R}(z),\tau_k} : k \in \mathbb{Z}, z \in \mathbb{S}^{2n - 1}} \cong \mathbb{S}^{2n - 1} \times \mathbb{Z},
	\end{equation*}
	\noindent for any $H \in \mathscr{F}_\varphi(\mathbb{S}^{2n - 1})$, where
	\begin{equation*}
		\tau_k := \frac{\pi}{m}(mk - 1).
	\end{equation*}

	By Proposition \ref{prop:kernel_hessian_contact}, $(z_0,\eta) \in T_z\mathbb{S}^{2n - 1} \times \mathbb{R}$ belongs to the kernel of the Hessian at $(z,k) \in \Crit \mathscr{A}^H_\varphi$ if and only if $\eta = 0$ and
	\begin{equation*}
		z_0 \in \ker \del[1]{D(\varphi \circ \phi^R_{-\tau_k})\vert_z - \id_{T_z\mathbb{S}^{2n - 1}}}.
	\end{equation*}
	A direct computation yields $D(\varphi \circ \phi^R_{-\tau_k})\vert_z = \id_{T_z \mathbb{S}^{2n - 1}}$ and thus
	\begin{equation*}
		T_{(z,k)}\Crit \mathscr{A}^H_\varphi = T_z \mathbb{S}^{2n - 1} \times \cbr[0]{0} \cong \ker \Hess \mathscr{A}^H_\varphi \vert_{(z,k)}.
	\end{equation*}
	So the twisted Rabinowitz action functional $\mathscr{A}^H_\varphi$ is Morse--Bott with spheres.

	The full Conley--Zehnder index \cite[Definition~10.4.1]{frauenfelderkoert:3bp:2018} gives rise to a locally constant function
	\begin{equation*}
		\hat{\mu}_{\mathrm{CZ}} \colon \Crit \mathscr{A}^H_\varphi \to \mathbb{Z}, \qquad \hat{\mu}(z,k) = (2k-1)n.
	\end{equation*}
	Note that the definition of the Conley--Zehnder index also applies in this degenerate case, compare \cite[Remark~10.4.2]{frauenfelderkoert:3bp:2018}. By the adapted proof of the Hofer--Wysocki--Zehnder Theorem \cite[Theorem~12.2.1]{frauenfelderkoert:3bp:2018} to the $n$-dimensional setting, the full Conley--Zehnder index coincides with the transverse Conley--Zehnder index $\mu_{\CZ}$. Indeed, for $(\gamma,\tau) \in \Crit \mathscr{A}^H_\varphi$ define a smooth path
	\begin{equation*}
		\Psi \colon I \to \Sp(n), \qquad \Psi_t := D\phi^H_{\tau t}\vert_{\gamma(0)} \colon \mathbb{C}^n \to \mathbb{C}^n.
	\end{equation*}
	Adapting the proof of \cite[Lemma~12.2.3~(iii)]{frauenfelderkoert:3bp:2018}, we get that
	\begin{equation*}
		\Psi_1(R(\gamma(0))) = R(\gamma(1)) \qquad \text{and} \qquad \Psi_1(\gamma(0)) = \gamma(1).
	\end{equation*}
	Arguing as in \cite[p.~235--236]{frauenfelderkoert:3bp:2018} we conclude
	\begin{equation*}
		\mu_{\CZ}(\gamma,\tau) = \hat{\mu}_{\CZ}(\gamma,\tau).
	\end{equation*}
	
	Fix $z_0 \in \mathbb{S}^{2n - 1}$ and define $\eta := \phi^{\tau_0 R}(z_0)$. Note that $\phi^{\tau_k R}(z)$ belongs to the same equivalence class in $\pi_0 \mathscr{L}_\varphi \mathbb{S}^{2n - 1}$ as $\eta$ for all $z \in \mathbb{S}^{2n - 1}$ and $k \in \mathbb{Z}$ because $\mathbb{S}^{2n - 1}$ is simply connected for $n\geq 2$. Let $h \in C^\infty(\mathbb{S}^{2n - 1})$ be the standard height function. By Remark \ref{rem:grading}, $\RFH^\varphi(\mathbb{S}^{2n - 1},\mathbb{C}^n)$ carries the $\mathbb{Z}$-grading
	\begin{equation*}
		\mu((z,k),(z_0,0)) + \ind_h(z) = 2kn + \ind_h(z) \qquad \forall (z,k) \in \mathbb{S}^{2n - 1} \times \mathbb{Z}.
	\end{equation*}
	We claim that the number of twisted negative gradient flow lines between the minimum of $\mathbb{S}^{2n - 1} \times \{k + 1\}$ and the maximum of $\mathbb{S}^{2n - 1} \times \{k\}$ must be odd, so that the critical manifold $\Crit \mathscr{A}^H_\varphi$ looks like a string of pearls, see Figure \ref{fig:equivariant_rfh_1}. Indeed, if there is an even number of such negative gradient flow lines, then $\RFH^\varphi(\mathbb{S}^{2n - 1},\mathbb{C}^n) \neq 0$, contradicting Theorem \ref{thm:displaceable} as $\mathbb{S}^{2n - 1}$ is displaceable in the completion $\mathbb{C}^n$.

\begin{figure}[h!tb]
	\centering
	\includegraphics[width=\textwidth]{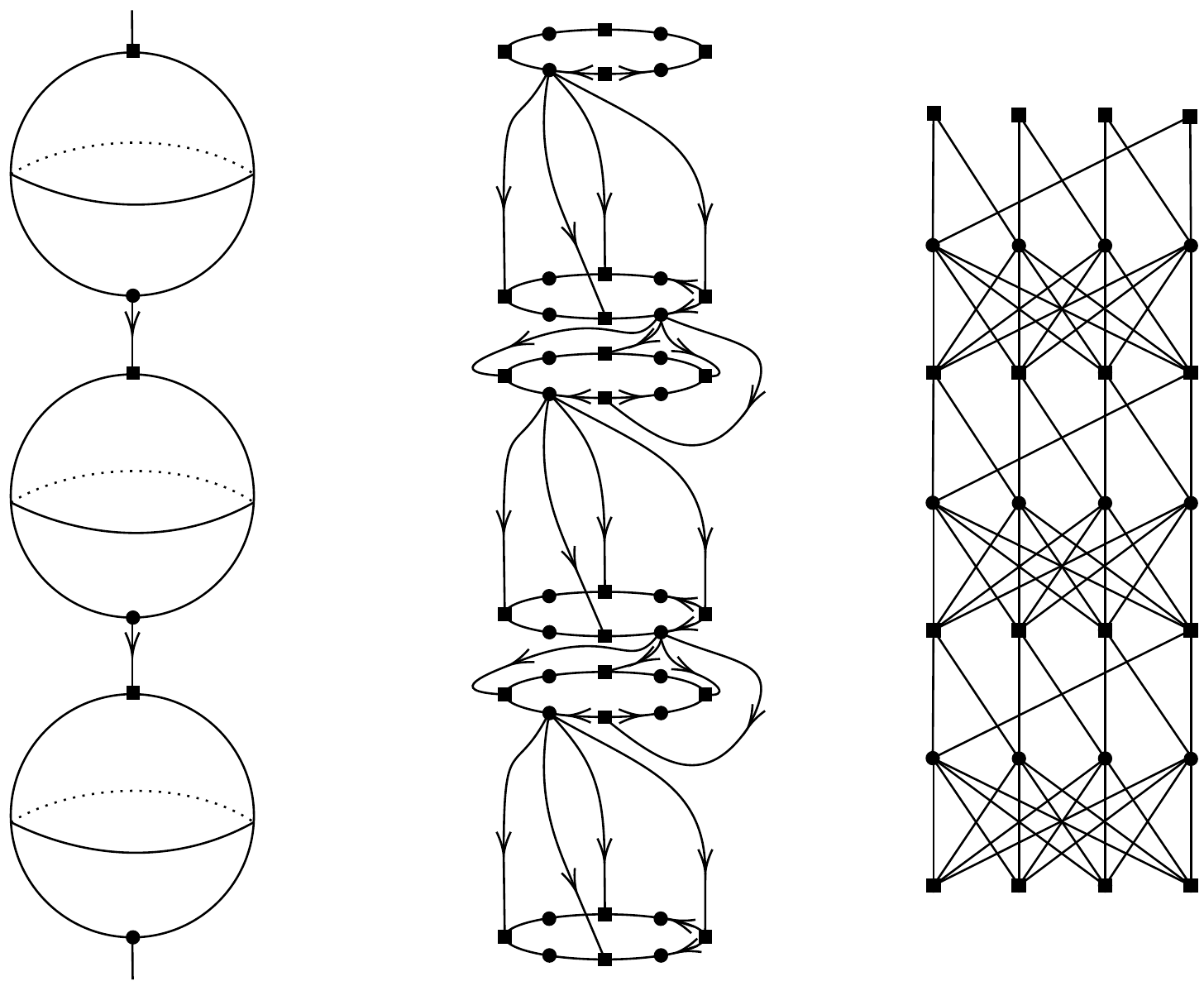}
	\caption{The critical manifold $\mathbb{S}^{2n - 1} \times \mathbb{Z}$ with the standard height function, the Morse--Bott function $f$ and the resulting chain complex.}
	\label{fig:equivariant_rfh_1}
\end{figure}

	To compute the $\mathbb{Z}_m$-equivariant twisted Rabinowitz--Floer homology, choose the additional $\mathbb{Z}_m$-invariant Morse--Bott function
	\begin{equation*}
		f \colon \mathbb{S}^{2n - 1} \to \mathbb{R}, \qquad f(z^1,\dots,z^n) := \sum_{j = 1}^n j\abs[0]{z^j}^2
	\end{equation*}
	\noindent on each component of $\Crit \mathscr{A}^H_\varphi$. It is easy to check that $f$ is Morse--Bott with circles. Additionaly, choose a $\mathbb{Z}_m$-invariant Morse function on $\Crit f$. 

	\noindent For example, one can take
	\begin{equation*}
		h \colon \mathbb{S}^1 \to \mathbb{R}, \qquad h(t) := \cos(2\pi m t).
	\end{equation*}
	The resulting chain complex is given by
	\begin{equation*}
		\begin{tikzcd}[column sep = scriptsize]
			\dots \arrow[r] & \mathbb{Z}_2^m \arrow[r,"\mathbbm{1}"] & \mathbb{Z}^m_2 \arrow[r,"A"] & \mathbb{Z}^m_2 \arrow[r,"\mathbbm{1}"] & \mathbb{Z}^m_2 \arrow[r,"A"] & \mathbb{Z}^m_2 \arrow[r,"\mathbbm{1}"] & \mathbb{Z}^m_2 \arrow[r] & \dots 
		\end{tikzcd}
	\end{equation*}
	\noindent where $\mathbbm{1} \in M_{m \times m}(\mathbb{Z}_2)$ has every entry equal to $1$ and $A \in M_{m \times m}(\mathbb{Z}_2)$ is defined by
	\begin{equation*}
		A := I_{m \times m} + \sum_{j = 1}^{m - 1} e_{(j + 1)j} + e_{1 m},
	\end{equation*}
	\noindent where $e_{ij} \in M_{m \times m}(\mathbb{Z}_2)$ satisfies $(e_{ij})_{kl} = \delta_{ik}\delta_{jl}$. Thus the resulting chain complex looks like a rope ladder. Compare Figure \ref{fig:equivariant_rfh_1}.

	Passing to the quotient via the free $\mathbb{Z}_m$-action, we get the acyclic chain complex
	\begin{equation*}
		\begin{tikzcd}
			\dots \arrow[r] & \mathbb{Z}_2 \arrow[r,"0"] & \mathbb{Z}_2 \arrow[r,"0"] & \mathbb{Z}_2 \arrow[r,"0"] & \mathbb{Z}_2 \arrow[r] & \dots 
		\end{tikzcd}
	\end{equation*}
	\noindent if $m$ is even and the alternating chain complex 
	\begin{equation*}
		\begin{tikzcd}[column sep = scriptsize]
			\dots \arrow[r] & \mathbb{Z}_2 \arrow[r,"1"] & \mathbb{Z}_2 \arrow[r,"0"] & \mathbb{Z}_2 \arrow[r,"1"] & \mathbb{Z}_2 \arrow[r,"0"] & \mathbb{Z}_2 \arrow[r,"1"] & \mathbb{Z}_2 \arrow[r] & \dots 
		\end{tikzcd}
	\end{equation*}
	\noindent if $m$ is odd. From this the statement follows.

\begin{figure}[h!tb]
	\centering
	\includegraphics[width=.45\textwidth]{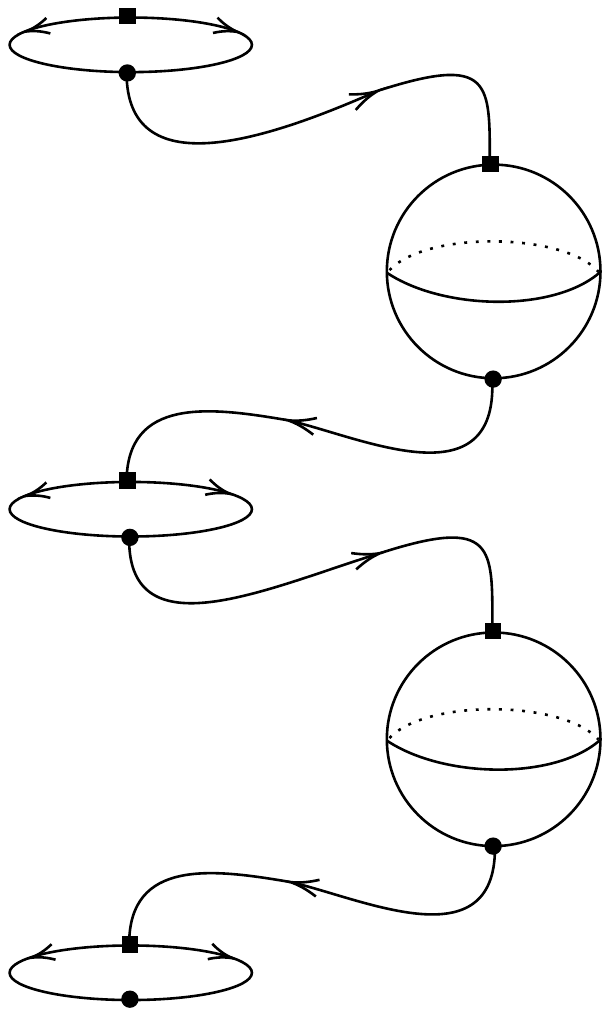}
	\caption{The critical manifold $\Crit \mathscr{A}^H_\varphi$ together with the standard height function.}
	\label{fig:equivariant_rfh_2}
\end{figure}

	For the general case, we note that $\Crit \mathscr{A}^H_\varphi$ is a disjoint union of spheres and different copies of $\mathbb{Z}$. It is therefore rather hard to compute the equivariant twisted Rabinowitz--Floer homology directly via analysing the critical manifold. By Theorem \ref{thm:displaceable}, there is a canonical isomorphism
	\begin{equation*}
		\RFH^\varphi_*(\mathbb{S}^{2n - 1},\mathbb{C}^n) \cong \RFH_*(\mathbb{S}^{2n - 1},\mathbb{C}^n)
	\end{equation*}
	\noindent inducing a canonical isomorphism
	\begin{equation*}
		\overline{\RFH}^\varphi_*(\mathbb{S}^{2n - 1}/\mathbb{Z}_m) \cong \RFH^{\mathbb{Z}_m}_*(\mathbb{S}^{2n - 1},\mathbb{C}^n),
	\end{equation*}
	\noindent where $\RFH^{\mathbb{Z}_m}_*(\mathbb{S}^{2n - 1},\mathbb{C}^n)$ denotes the $\mathbb{Z}_m$-equivariant Rabinowitz--Floer homology constructed in \cite[p.~487]{albersfrauenfelder:eh:2012}. A computation similar as before shows
	\begin{equation*}
		\RFH^{\mathbb{Z}_m}_k(\mathbb{S}^{2n - 1},\mathbb{C}^n) \cong \begin{cases}
			\mathbb{Z}_2 & m \text{ even},\\
			0 & m \text{ odd},
		\end{cases} \qquad \forall k \in \mathbb{Z}.
	\end{equation*}
	The crucial observation is, that $\Crit \mathscr{A}^H \cong \mathbb{S}^{2n - 1} \times \mathbb{Z}$. In particular, the string of pearl looks like in Figure \ref{fig:equivariant_rfh_2}.

	Finally, $\overline{\RFH}^\varphi_k(\mathbb{S}^{2n - 1}/\mathbb{Z}_m)$ is generated by a noncontractible periodic Reeb orbit in $\mathbb{S}^{2n - 1}/\mathbb{Z}_m$ for all $k \in \mathbb{Z}$ by Lemma \ref{lem:noncontractible}.
\end{proof}

\begin{remark}[Coefficients]
	As $\overline{\RFH}^\varphi_*(\mathbb{S}^{2n - 1}/\mathbb{Z}_m)$ vanishes for odd $m$, one should rather consider twisted Rabinowitz--Floer homology with coefficients in $\mathbb{Z}$ in that case.
\end{remark}

For an immediate algebraic corollary recall the definition of Tate cohomology \cite[Definition~6.2.4]{weibel:homological_algebra:1994} and Tate homology \cite[p.~135]{brown:groups:1982}.

\begin{corollary}[Tate Homology]
	Let $C_m$ denote the cyclic group of order $m \geq 1$. Then for the trivial left $C_m$-module $\mathbb{Z}_2$ we have that
	\begin{equation*}
		\overline{\RFH}^\varphi_k(\mathbb{S}^{2n - 1}/\mathbb{Z}_m) \cong \hat{\operatorname{H}}_k(C_m;\mathbb{Z}_2) \qquad \forall k \in \mathbb{Z},
	\end{equation*}
	\noindent where $\hat{\operatorname{H}}_*(C_m;\mathbb{Z}_2)$ denotes the Tate homology group of $C_m$ with coefficients in the trivial left $C_m$-module $\mathbb{Z}_2$.
\end{corollary}

Using Theorem \ref{thm:equivariant_twisted_rfh} we can prove Theorem \ref{thm:noncontractible}.

\begin{theorem}
	Let $\Sigma \subseteq \mathbb{C}^n$, $n \geq 2$, be a compact and connected star-shaped hypersurface invariant under the rotation	
	\begin{equation*}
		\varphi \colon \mathbb{C}^n \to \mathbb{C}^n, \quad \varphi(z^1,\dots,z^n) := \del[1]{e^{2\pi i k_1/m}z^1,\dots,e^{2\pi i k_n/m}z^n}
	\end{equation*}
	\noindent for some even $m \geq 2$ and $k_1,\dots,k_n \in \mathbb{Z}$ coprime to $m$. Then $\Sigma/\mathbb{Z}_m$ admits a noncontractible periodic Reeb orbit.	
\end{theorem}

\begin{proof}
	By assumption, $\Sigma$ bounds a star-shaped domain $D$ with respect to the origin. Thus $(D \cup \Sigma,d\lambda)$ is a Liouville domain with $\lambda$ given by \eqref{eq:Liouville_form}. By rescaling we may assume that $\mathbb{S}^{2n - 1} \subseteq D$. Define a smooth function
\begin{equation*}
	\delta \colon \Sigma \to \intoo[0]{-\infty,0}
\end{equation*}
\noindent by requiring $\delta(x)$ to be the unique number such that $\phi^X_{\delta(x)}(x) \in \mathbb{S}^{2n - 1}$, $x \in \Sigma$, where
	\begin{equation*}
		X = \frac{1}{2}\del[3]{x^j\frac{\partial}{\partial x^j} + y^j \frac{\partial}{\partial y^j}}
	\end{equation*}
	\noindent denotes the Liouville vector field. We claim that $\delta \circ \varphi = \delta$. Indeed, $\delta(\varphi(x))$ is the unique number such that $\phi^X_{\delta(\varphi(x))}(\varphi(x)) \in \mathbb{S}^{2n - 1}$. As the flow of $X$ and $\varphi$ commute by the proof of Lemma \ref{lem:defining_Hamiltonian}, we conclude that $\phi^X_{\delta(\varphi(x))}(x) \in \mathbb{S}^{2n - 1}$. Define a smooth family of star-shaped hypersurfaces $(\Sigma_\sigma)_{\sigma \in I}$
	\begin{equation*}
		\Sigma_\sigma := \cbr[1]{\phi^X_{\sigma\delta(x)}(x) : x \in \Sigma} \subseteq \mathbb{C}^n.
	\end{equation*}
	Then we compute
	\begin{align*}
		\varphi(\Sigma_\sigma) &= \cbr[1]{\varphi\del[1]{\phi^X_{\sigma\delta(x)}(x)} : x \in \Sigma}\\
		&= \cbr[1]{\phi^X_{\sigma\delta(x)}(\varphi(x)) : x \in \Sigma}\\
		&= \cbr[1]{\phi^X_{\sigma\delta(\varphi(x))}(\varphi(x)) : x \in \Sigma}\\
		&= \cbr[1]{\phi^X_{\sigma\delta(y)}(y) : y \in \varphi(\Sigma)}\\
		&= \cbr[1]{\phi^X_{\sigma\delta(y)}(y) : y \in \Sigma}\\
		&= \Sigma_\sigma
	\end{align*}
	\noindent for all $\sigma \in I$ and therefore we can find a twisted homotopy $(H_\sigma)_{\sigma \in I}$ of Liouville domains in $\mathbb{C}^n$. By Theorem \ref{thm:invariance} we have that
	\begin{equation*}
		\RFH^\varphi_*(\Sigma,\mathbb{C}^n) \cong \RFH_*^\varphi(\mathbb{S}^{2n - 1},\mathbb{C}^n),
	\end{equation*}
	\noindent giving rise to a canonical isomorphism of the associated $\mathbb{Z}_m$-equivariant twisted Rabinowitz--Floer homology
	\begin{equation*}
		\overline{\RFH}^\varphi_*(\Sigma/\mathbb{Z}_m) \cong \overline{\RFH}^\varphi_*(\mathbb{S}^{2n - 1}/\mathbb{Z}_m). 
	\end{equation*}
	But by Theorem \ref{thm:equivariant_twisted_rfh} the latter does not vanish as $m \geq 2$ is even.
\end{proof}

\section*{Acknowledgements}
	First of all I would like to thank my supervisor Urs Frauenfelder for his inspiring guidance. I owe thanks to Kai Cieliebak and Igor Uljarevic for many helpful discussions. Lastly, I also thank Will J. Merry for his constant encouragement during my Master's thesis as well as Felix Schlenk and the anonymous referee for improving the exposition of this paper.

\begin{appendix}
	\section{Twisted Loops in Universal Covering Manifolds}
\label{twisted_loops_on_universal_covering_manifolds}

In this Appendix, we will consider the category of topological manifolds rather than the category of smooth manifolds, because smoothness does not add much to the discussion. Free and based loop spaces are fundamental objects in Algebraic Topology, for a vast treatment of the geometry and topology of based as well as free loop spaces see for example \cite{loop_spaces:2015}. But so-called twisted loop spaces are not considered that much. 

\begin{theorem}[Twisted Loops in Universal Covering Manifolds]
	Let $(M,x)$ be a connected pointed topological manifold and $\pi \colon \tilde{M} \to M$ the universal covering. 
	\begin{enumerate}[label=\textup{(\alph*)}]
		\item Fix $[\eta] \in \pi_1(M,x)$ and denote by $U_\eta \subseteq \mathscr{L}(M,x)$ the path component corresponding to $[\eta]$ via the bijection $\pi_0(\mathscr{L}(M,x)) \cong \pi_1(M,x)$. For every $e,e' \in \pi^{-1}(x)$ and $\varphi \in \Aut_\pi(\tilde{M})$ such that $\varphi(e) = \tilde{\eta}_e(1)$, where $\tilde{\eta}_e$ denotes the unique lift of $\eta$ with $\tilde{\eta}_e(0) = e$, we have a commutative diagram of homeomorphisms
		\begin{equation}
			\label{cd:twisted}
			\qquad \begin{tikzcd}
				\mathscr{L}_\varphi(\tilde{M},e) \arrow[rr,"L_\psi"] & & \mathscr{L}_{\psi \circ \varphi \circ \psi^{-1}}(\tilde{M},e')\\
				& U_\eta \arrow[lu,"\Psi_e"] \arrow[ru,"\Psi_{e'}"'],
			\end{tikzcd}
		\end{equation}
		\noindent where $\psi \in \Aut_\pi(\tilde{M})$ is such that $\psi(e) = e'$, 
		\begin{equation*}
			\qquad L_\psi \colon \mathscr{L}_\varphi(\tilde{M},e) \to \mathscr{L}_{\psi \circ \varphi \circ \psi^{-1}}(\tilde{M},e'), \qquad L_\psi(\gamma) := \psi \circ \gamma,
		\end{equation*}
		\noindent and
		\begin{align*}
			\qquad& \Psi_e \colon U_\eta \to \mathscr{L}_\varphi(\tilde{M},e), & \Psi_e(\gamma) := \tilde{\gamma}_e,\\
			\qquad& \Psi_{e'} \colon U_\eta \to \mathscr{L}_{\psi \circ \varphi \circ \psi^{-1}}(\tilde{M},e'), & \Psi_{e'}(\gamma) := \tilde{\gamma}_{e'}.
		\end{align*}
		Moreover, $U_{c_x} \cong \mathscr{L}_\varphi(\tilde{M},e)$ via $\Psi_e$ if and only if $\varphi = \id_{\tilde{M}}$, where $c_x$ denotes the constant loop at $x$.
	\item For every $\varphi \in \Aut_\pi(\tilde{M})$ and $e,e' \in \pi^{-1}(x)$ we have a commutative diagram of isomorphisms
	\begin{equation*}
		\qquad \begin{tikzcd}
		\Aut_\pi(\tilde{M}) \arrow[rr,"C_\psi"] & & \Aut_\pi(\tilde{M})\\
		& \pi_1(M,x) \arrow[lu,"\Phi_e"] \arrow[ru,"\Phi_{e'}"'],
	\end{tikzcd}
	\end{equation*}
	\noindent where for $\psi \in \Aut_\pi(\tilde{M})$ sucht that $\psi(e) = e'$
	\begin{equation*}
		\qquad C_\psi \colon \Aut_\pi(\tilde{M}) \to \Aut_\pi(\tilde{M}), \qquad C_\psi(\varphi) := \psi \circ \varphi \circ \psi^{-1},
	\end{equation*}
	\noindent and
	\begin{align*}
		\qquad & \Phi_e \colon \pi_1(M,x) \to \Aut_\pi(\tilde{M}), & \Phi_e([\gamma]) := \varphi^e_{[\gamma]},\\
		\qquad & \Phi_{e'} \colon \pi_1(M,x) \to \Aut_\pi(\tilde{M}), & \Phi_{e'}([\gamma]) := \varphi^{e'}_{[\gamma]},
	\end{align*}
	\noindent with $\varphi^e_{[\gamma]}(e) = \tilde{\gamma}_e(1)$ and $\varphi^{e'}_{[\gamma]}(e') = \tilde{\gamma}_{e'}(1)$.
\item The projection
	\begin{equation*}
		\qquad \tilde{\pi}_x \colon \coprod_{\substack{\varphi \in \Aut_\pi(\tilde{M})\\e \in \pi^{-1}(x)}} \mathscr{L}_\varphi(\tilde{M},e) \to \mathscr{L}(M,x)
	\end{equation*}
	\noindent defined by $\tilde{\pi}_x(\gamma) := \pi \circ \gamma$ is a covering map with number of sheets coinciding with the cardinality of $\pi_1(M,x)$. Moreover, $\tilde{\pi}_x$ restricts to define a covering map
	\begin{equation*}
		\qquad \tilde{\pi}_x\vert_{\id_{\tilde{M}}} \colon \coprod_{e \in \pi^{-1}(x)} \mathscr{L}(\tilde{M},e) \to U_{c_x},
	\end{equation*}
	\noindent and $\tilde{\pi}_x$ gives rise to a principal $\Aut_\pi(\tilde{M})$-bundle. If $M$ admits a smooth structure, then this bundle is additionally a bundle of smooth Banach manifolds.
\end{enumerate}
	\label{thm:cd_twisted}
\end{theorem} 

\begin{proof}
	For proving part (a), fix a path class $[\gamma] \in \pi_1(M,x)$. As any topological manifold is Hausdorff, paracompact and locally metrisable by definition, the Smirnov Metrisation Theorem \cite[Theorem~42.1]{munkres:topology:2000} implies that $M$ is metrisable. Let $d$ be a metric on $M$ and $\bar{d}$ be the standard bounded metric corresponding to $d$, that is,
	\begin{equation*}
		\bar{d}(x,y) = \min\cbr[0]{d(x,y),1} \qquad \forall x,y \in M.
	\end{equation*}
	The metric $\bar{d}$ induces the same topology on $M$ as $d$ by \cite[Theorem~20.1]{munkres:topology:2000}. Topologise the based loop space $\mathscr{L}(M,x) \subseteq \mathscr{L}M$ as a subspace of the free loop space on $M$, where $\mathscr{L}M$ is equipped with the topology of uniform convergence, that is, with the supremum metric
	\begin{equation*}
		\bar{d}_\infty(\gamma,\gamma') = \sup_{t \in \mathbb{S}^1}\bar{d}\del[1]{\gamma(t),\gamma'(t)} \qquad \forall \gamma,\gamma' \in \mathscr{L}M.
	\end{equation*}
	There is a canonical pseudometric on the universal covering manifold $\tilde{M}$ induced by $\bar{d}$ given by $\bar{d} \circ \pi$. As every pseudometric generates a topology, we topologise the based twisted loop space $\mathscr{L}_\varphi(\tilde{M},e) \subseteq \mathscr{P}\tilde{M}$ as a subspace of the free path space on $\tilde{M}$ for every $e \in \pi^{-1}(x)$ via the supremum metric $\tilde{d}_\infty$ corresponding to $\bar{d} \circ \pi$. In fact, $\tilde{d}_\infty$ is a metric as if $\tilde{d}_\infty(\gamma,\gamma') = 0$, then by definition of $\tilde{d}_\infty$ we have that $\pi(\gamma) = \pi(\gamma')$. But as $\gamma(0) = e = \gamma'(0)$, we conclude $\gamma = \gamma'$ by the unique lifting property of paths \cite[Corollary~11.14]{lee:tm:2011}. Note that the resulting topology of uniform convergence on $\mathscr{L}_\varphi(\tilde{M},e)$ coincides with the compact-open topology by \cite[Theorem~46.8]{munkres:topology:2000} or \cite[Proposition~A.13]{hatcher:at:2001}. In particular, the topology of uniform convergence does not depend on the choice of a metric (see \cite[Corollary~46.9]{munkres:topology:2000}). It follows from \cite[Theorem~11.15~(b)]{lee:tm:2011}, that $\Psi_e$ and $\Psi_{e'}$ are well-defined. Moreover, it is immediate by the fact that the projection $\pi \colon \tilde{M} \to M$ is an isometry with respect to the above metric, that $\Psi_e$ and $\Psi_{e'}$ are continuous with continuous inverse given by the composition with $\pi$. It is also immediate that $L_\psi$ is continuous with continuous inverse $L_{\psi^{-1}}$.

	Next we show that the diagram \eqref{cd:twisted} commutes. Note that
	\begin{equation*}
		\pi \circ L_\psi \circ \Psi_e = \pi \circ \Psi_e = \id_{U_\eta} = \pi \circ \Psi_{e'},
	\end{equation*}
	\noindent thus by
	\begin{equation*}
		(L_\psi \circ \Psi_e(\gamma))(0) = \psi(\tilde{\gamma}_e(0)) = \psi(e) = e' = \tilde{\gamma}_{e'}(0) = \Psi_{e'}(\gamma)(0)
	\end{equation*}
	\noindent and by uniqueness it follows that 
	\begin{equation*}
		L_\psi \circ \Psi_e = \Psi_{e'}.
	\end{equation*}
	In particular
	\begin{equation*}
		\Psi_{e'}(1) = (L_\psi \circ \Psi_e)(1) = \psi(\varphi(e)) = (\psi \circ \varphi \circ \psi^{-1})(e'),
	\end{equation*}
	\noindent and thus $\Psi_{e'}(\gamma) \in \mathscr{L}_{\psi \circ \varphi \circ \psi^{-1}}(\tilde{M},e')$. Consequently, the homeomorphism $\Psi_{e'}$ is well-defined.

	Recall, that by the Monodromy Theorem \cite[Theorem~11.15~(b)]{lee:tm:2011} 
	\begin{equation*}
		\gamma \simeq \gamma' \qquad \Leftrightarrow \qquad \Psi_e(\gamma)(1) = \Psi_e(\gamma')(1)
	\end{equation*}
	\noindent for all paths $\gamma$ and $\gamma'$ in $M$ starting at $x$ and ending at the same point. Note that the statement of the the Monodromy Theorem is an if-and-only-if statement since $\tilde{M}$ is simply connected.

	Suppose $\gamma \in \mathscr{L}(M,x)$ is contractible. Then $\gamma \simeq c_x$, implying $e \in \Fix(\varphi)$. But the only deck transformation of $\pi$ fixing any point of $\tilde{M}$ is $\id_{\tilde{M}}$ by \cite[Proposition~12.1~(a)]{lee:tm:2011}.

	Conversely, assume that $\gamma \in \mathscr{L}(M,x)$ is not contractible. Then we have that $\Psi_e(\gamma)(1) \neq e$. Indeed, if $\Psi_e(\gamma)(1) = e$, then $\gamma \simeq c_x$ and consequently, $\gamma$ would be contractible. As normal covering maps have transitive automorphism groups by \cite[Corollary~12.5]{lee:tm:2011}, there exists $\psi \in \Aut_\pi(\tilde{M}) \setminus \cbr[0]{\id_{\tilde{M}}}$ such that $\Psi_e(\gamma)(1) = \psi(e)$.  

	For proving part (b), observe that $\Phi_e$ and $\Phi_{e'}$ are isomorphisms follows from \cite[Corollary~12.9]{lee:tm:2011}. Moreover, it is also clear that $C_\psi$ is an isomorphism with inverse $C_{\psi^{-1}}$. Let $[\gamma] \in \pi_1(M,x)$. Then using part (a) we compute
\allowdisplaybreaks
	\begin{align*}
	(C_\psi \circ \Phi_e)[\gamma](e') &= (\psi \circ \Phi_e[\gamma] \circ \psi^{-1})(e')\\
	&= \psi\del[1]{\varphi_{[\gamma]}^e(e)}\\
	&= \psi(\tilde{\gamma}_e(1))\\
	&= (L_\psi \circ \Psi_e)(\gamma)(1)\\
	&= \Psi_{e'}(\gamma)(1)\\
	&= \tilde{\gamma}_{e'}(1)\\
	&= \varphi^{e'}_{[\gamma]}(e')\\
	&= \Phi_{e'}[\gamma](e').
\end{align*}
Thus by uniqueness \cite[Proposition~12.1~(a)]{lee:tm:2011}, we conclude
\begin{equation*}
	C_\psi \circ \Phi_e = \Phi_{e'}.
\end{equation*}
Finally for proving (c), define a metric $\tilde{d}_\infty$ on 
	\begin{equation*}
		E := \coprod_{\substack{\varphi \in \Aut_\pi(\tilde{M})\\e \in \pi^{-1}(x)}} \mathscr{L}_\varphi(\tilde{M},e)
	\end{equation*}
	\noindent by
	\begin{equation*}
		\tilde{d}_\infty(\gamma,\gamma') := \begin{cases}
			\bar{d}_\infty\del[1]{\pi(\gamma),\pi(\gamma')} & \gamma,\gamma' \in \mathscr{L}_\varphi(\tilde{M},e),\\
			1 & \text{else}.
		\end{cases}
	\end{equation*}
	Then the induced topology coincides with the disjoint union topology and with respect to this topology, $\tilde{\pi}_x$ is continuous. So left to show is that $\tilde{\pi}_x$ is a covering map. Surjectivity is clear. So let $\gamma \in \mathscr{L}(M,x)$. Then $\gamma \in U_\eta$ for some $[\eta] \in \pi_1(M,x)$. Now note that $U_\eta$ is open in $\mathscr{L}(M,x)$ and by part (a) we conclude
	\begin{equation}
		\label{eq:twisted_fibre}
		\tilde{\pi}_x^{-1}(U_\eta) = \coprod_{\psi \in \Aut_\pi(\tilde{M})} \mathscr{L}_{\psi \circ \varphi \circ \psi^{-1}}(\tilde{M},\psi(e))
	\end{equation}
	\noindent for some fixed $e \in \pi^{-1}(x)$ and $\varphi \in \Aut_\pi(\tilde{M})$ such that $\varphi(e) = \tilde{\eta}_e(1)$. 


	As the cardinality of the fibre $\pi^{-1}(x)$ and of $\Aut_\pi(\tilde{M})$ coincides with the cardinality of the fundamental group $\pi_1(M,x)$ by \cite[Corollary~11.31]{lee:tm:2011} and part (b), we conclude that the number of sheets is equal to the cardinality of the fundamental group $\pi_1(M,x)$.

  Equip $\Aut_\pi(\tilde{M})$ with the discrete topology. As the fundamental group of every topological manifold is countable by \cite[Theorem~7.21]{lee:tm:2011}, we have that $\Aut_\pi(\tilde{M})$ is a discrete topological Lie group. Now label the distinct path classes in $\pi_1(M,x)$ by $\beta \in B$ and for fixed $e \in \pi^{-1}(x)$ define local trivialisations
 \begin{equation*}
	 (\tilde{\pi}_x,\alpha_\beta) \colon \tilde{\pi}_x^{-1}(U_\beta)  \xrightarrow{\cong} U_\beta \times \Aut_\pi(\tilde{M}),
 \end{equation*} 
 \noindent making use of \eqref{eq:twisted_fibre} by
 \begin{equation*}
	 \alpha_\beta(\gamma) := \psi^{-1},
 \end{equation*}
 \noindent whenever $\gamma \in \mathscr{L}_{\psi \circ \varphi \circ \psi^{-1}}(\tilde{M},\psi(e))$. Consequently, $\tilde{\pi}_x$ is a fibre bundle with discrete fibre $\Aut_\pi(\tilde{M})$ and bundle atlas $(U_\beta,\alpha_\beta)_{\beta \in B}$. Define a free right action
 \begin{equation*}
	 E \times \Aut_\pi(\tilde{M}) \to E, \qquad \gamma \cdot \xi := \xi^{-1} \circ \gamma.
 \end{equation*}
  Then $\alpha_\beta$ is $\Aut_\pi(\tilde{M})$-equivariant with respect to this action for all $\beta \in B$. Indeed, using again the commutative diagram \eqref{cd:twisted} we compute
 \begin{equation*}
	 \alpha_\beta(\gamma \cdot \xi) = \alpha_\beta(\xi^{-1} \circ \gamma) = \del[1]{\xi^{-1} \circ \psi}^{-1} = \psi^{-1} \circ \xi = \alpha_\beta(\gamma) \circ \xi
 \end{equation*}
 \noindent for all $\xi \in \Aut_\pi(\tilde{M})$ and $\gamma \in \mathscr{L}_{\psi \circ \varphi \circ \psi^{-1}}(\tilde{M},\psi(e))$. Note, that here we use again the fact that $\Aut_\pi(\tilde{M})$ acts transitively on the fibre $\pi^{-1}(x)$.

 Suppose that $M$ admits a smooth structure. Then for every compact smooth manifold $N$ we have that the mapping space $C(N,M)$ admits the structure of a smooth Banach manifold by \cite{wittmann:loop_space:2019}. By \cite[Theorem~1.1~p.~24]{loop_spaces:2015}, there is a smooth fibre bundle, called the \bld{loop-loop fibre bundle},
 \begin{equation*}
		 \mathscr{L}(M,x) \hookrightarrow \mathscr{L}M \xrightarrow{\ev_0} M 
 \end{equation*}
 \noindent where
 \begin{equation*}
	 \ev_0 \colon \mathscr{L} M \to M, \qquad \ev_0(\gamma) := \gamma(0).
 \end{equation*}
 Thus the based loop space $\mathscr{L}(M,x) = \ev_0^{-1}(x)$ on $M$ is a smooth Banach manifold by the implicit function theorem \cite[Theorem~A.3.3]{mcduffsalamon:J-holomorphic_curves:2012} for all $x \in M$. Likewise, by \cite[Theorem~1.2~p.~25]{loop_spaces:2015}, there is a smooth fibre bundle, called the \bld{path-loop fibre bundle}, 
  \begin{equation*}
	  \mathscr{L}(\tilde{M},e) \hookrightarrow \mathscr{P}(\tilde{M},e) \xrightarrow{\ev_1} \tilde{M}, 
 \end{equation*}
 \noindent where
 \begin{equation*}
	 \mathscr{P}(\tilde{M},e) := \{\gamma \in C(I,\tilde{M}) : \gamma(0) = e\}
 \end{equation*}
 \noindent denotes the based path space and
 \begin{equation*}
	 \ev_1 \colon \mathscr{P}(\tilde{M},e) \to \tilde{M}, \qquad \ev_1(\gamma) := \gamma(1).
 \end{equation*}
 Therefore, the twisted loop space $\mathscr{L}_\varphi(\tilde{M},e) = \ev_1^{-1}(\varphi(e))$ is also a smooth Banach manifold for all $\varphi \in \Aut_\pi(\tilde{M})$ and $e \in \pi^{-1}(x)$ by the implicit function theorem \cite[Theorem~A.3.3]{mcduffsalamon:J-holomorphic_curves:2012}. As the fundamental group $\pi_1(M,x)$ is countable, the topological space $E$ has only countably many connected components being smooth Banach manifolds and thus the total space itself is a smooth Banach manifold. Finally, $\Aut_\pi(\tilde{M})$ is trivially a Banach manifold with $\dim \Aut_\pi(\tilde{M}) = 0$ as a discrete Lie group.
\end{proof}

\begin{corollary}
	\label{cor:cd_twisted_abelian}
	Let $(M,x)$ be a connected pointed topological manifold and denote by $\pi \colon \tilde{M} \to M$ the universal covering of $M$. Assume that $\pi_1(M,x)$ is abelian.
	\begin{enumerate}[label=\textup{(\alph*)}]
		\item Fix a path class $[\eta] \in \pi_1(M,x)$. For every $e,e' \in \pi^{-1}(x)$ and deck transformation $\varphi \in \Aut_\pi(\tilde{M})$ such that $\varphi(e) = \tilde{\eta}_e(1)$, we have a commutative diagram of homeomorphisms
		\begin{equation*}
			\qquad \begin{tikzcd}
				\mathscr{L}_\varphi(\tilde{M},e) \arrow[rr,"L_\psi"] & & \mathscr{L}_\varphi(\tilde{M},e')\\
				& U_\eta \arrow[lu,"\Psi_e"] \arrow[ru,"\Psi_{e'}"'],
			\end{tikzcd}
		\end{equation*}
		\noindent where $\psi \in \Aut_\pi(\tilde{M})$ is such that $\psi(e) = e'$.
	\item For every $\varphi \in \Aut_\pi(\tilde{M})$ we have that $\Phi_e = \Phi_{e'}$ for all $e,e' \in \pi^{-1}(x)$. 
	\end{enumerate}
\end{corollary}

Lemma \ref{lem:noncontractible} now follows from part (a) of Theorem \ref{thm:cd_twisted}. Indeed, by assumption $\varphi \in \Aut_\pi(\Sigma) \setminus \{\id_\Sigma\}$ and using the long exact sequence of homotopy groups of a fibration \cite[Theorem~4.41]{hatcher:at:2001}, there is a short exact sequence 
\begin{equation*}
	\begin{tikzcd}[column sep = scriptsize]
		0 \arrow[r] & \pi_1(\Sigma,x) \arrow[r] & \pi_1(\Sigma/\mathbb{Z}_m,\pi(x)) \arrow[r] & \pi_0(\mathbb{Z}_m) \arrow[r] & 0.
	\end{tikzcd}
\end{equation*}
In particular, by \cite[Corollary~12.9]{lee:tm:2011} we conclude
\begin{equation*}
	\Aut_\pi(\Sigma) \cong \pi_1(\Sigma/\mathbb{Z}_m,\pi(x)) \cong \mathbb{Z}_m \cong \{\id_\Sigma,\varphi,\dots,\varphi^{m - 1}\}.
\end{equation*}

Finally, we discuss a smooth structure on the continuous free twisted loop space of a smooth manifold.


\begin{lemma}
	\label{lem:free_twisted_loop_space}
	Let $M$ be a smooth manifold and $\varphi \in \Diff(M)$. Then the continuous free twisted loop space $\mathscr{L}_\varphi M$ is the pullback of
	\begin{equation*}
		(\ev_0,\ev_1) \colon \mathscr{P}M \to M \times M, \qquad \gamma \mapsto (\gamma(0),\gamma(1)),
	\end{equation*}
	\noindent where we abbreviate $\mathscr{P}M := C(I,M)$, along the graph of $\varphi$
	\begin{equation*}
		\Gamma_\varphi \colon M \to M \times M, \qquad \Gamma_\varphi(x) := (x,\varphi(x)),
	\end{equation*}
	\noindent in the category of smooth Banach manifolds. Moreover, we have that
	\begin{equation*}
		T_\gamma \mathscr{L}_\varphi M = \{X \in \Gamma^0(\gamma^*TM) : X(1) = D\varphi(X(0))\}
	\end{equation*}
	\noindent for all $\gamma \in \mathscr{L}_\varphi M$.
\end{lemma}

\begin{proof}
	Write $f := (\ev_0,\ev_1)$. Then
	\begin{equation*}
		\mathscr{L}_\varphi M = f^{-1}(\Gamma_\varphi(M)).
	\end{equation*}
	Thus in order to show that the free twisted loop space $\mathscr{L}_\varphi M$ is a smooth Banach manifold, it is enough to show that $f$ is transverse to the properly embedded smooth submanifold $\Gamma_\varphi(M) \subseteq M \times M$. By \cite[Proposition~2.4]{lang:dg:1999} we need to show that the composition
	\begin{equation*}
		\Phi_\gamma\colon T_\gamma\mathscr{P}M \xrightarrow{Df_\gamma} T_{(x,\varphi(x))}(M \times M) \to T_{(x,\varphi(x))}(M \times M)/T_{(x,\varphi(x))}\Gamma_\varphi(M)
	\end{equation*}
	\noindent is surjective and $\ker \Phi_\gamma$ is complemented for all $\gamma \in f^{-1}(\Gamma_\varphi(M))$, where we abbreviate $x := \gamma(0)$. Note that we have a canonical isomorphism
	\begin{equation*}
		T_{(x,\varphi(x))}(M \times M)/T_{(x,\varphi(x))}\Gamma_\varphi(M) \to T_{\varphi(x)}M, \quad [(v,u)] := u - D\varphi(v).
	\end{equation*}
	Under this canonical isomorphism, the linear map $\Phi_\gamma$ is given by
	\begin{equation*}
		\Phi_\gamma(X) = X(1) - D\varphi(X(0)), \qquad \forall X \in \Gamma^0(\gamma^*TM).
	\end{equation*}
	Fix a Riemannian metric on $M$ and let $X_v \in \Gamma(\gamma^*TM)$ be the unique parallel vector field with $X_v(1) = v \in T_{\varphi(x)}M$. Fix a cutoff function $\beta \in C^\infty(I)$ such that $\supp \beta \subseteq \intcc[1]{\frac{1}{2},1}$ and $\beta = 1$ in a neighbourhood of $1$. Then $\Phi_\gamma(\beta X_v) = v$ and consequently, $\Phi_\gamma$ is surjective. Moreover
	\begin{equation*}
		\ker \Phi_\gamma = \cbr[0]{X \in \Gamma^0(\gamma^*TM) : X(1) = D\varphi(X(0))}
	\end{equation*}
	\noindent is complemented by the finite-dimensional vector space 
	\begin{equation*}
		V := \cbr[0]{\beta X_v \in \Gamma(\gamma^*TM) : v \in T_{\varphi(x)}M}.
	\end{equation*}
	Indeed, any $X \in \Gamma^0(\gamma^*TM)$ can be decomposed uniquely as
	\begin{equation*}
		X = X - \beta X_v + \beta X_v, \qquad v := X(1) - D\varphi(X(0)).
	\end{equation*}
	Abbreviating $Y := X - \beta X_v \in \Gamma^0(\gamma^*TM)$, we have that
	\begin{equation*}
		Y(1) = D\varphi(X(0)) = D\varphi(Y(0)),
	\end{equation*}
	\noindent implying $Y \in \ker \Phi_\gamma$. Thus $\mathscr{L}_\varphi M$ is a smooth Banach manifold.

	Now note that $\mathscr{L}_\varphi M$ can be identified with the pullback
	\begin{equation*}
		f^*\mathscr{P}M = \{(x,\gamma) \in M \times \mathscr{P}M : (\gamma(0),\gamma(1)) = \del[0]{x,\varphi(x)}\},
	\end{equation*}
	\noindent making the diagram 
	\begin{equation*}
		\begin{tikzcd}
			f^* \mathscr{P}M \arrow[r,"\pr_2"]\arrow[d,"\pr_1"'] & \mathscr{P}M \arrow[d,"f"]\\
			M \arrow[r,"\Gamma_\varphi"'] & M \times M
		\end{tikzcd}
	\end{equation*}
	\noindent commute, via the homeomorphism
	\begin{equation*}
		\mathscr{L}_\varphi M \to f^*\mathscr{P}M, \qquad \gamma \mapsto (\gamma(0),\gamma).
	\end{equation*}
	Finally, one computes
	\begin{equation*}
		T_{(x,\gamma)}f^*\mathscr{P}M = \{(v,X) \in T_xM \times T_\gamma \mathscr{P}M : Df_\gamma X = D\Gamma_\varphi\vert_x(v)\}
	\end{equation*}
	\noindent for all $(x,\gamma) \in f^*\mathscr{P}M$.
\end{proof}

\begin{remark}
	Using Lemma \ref{lem:free_twisted_loop_space} one should be able to prove similar results as in Theorem \ref{thm:cd_twisted} in the case of free twisted loop spaces. However, in the non-abelian case the situation gets much more complicated as in general it is not true, that lifts of conjugated elements of the fundamental group lie in the same free twisted loop space by \cite[Theorem~1.6~(i)]{loop_spaces:2015}.
\end{remark}
	
\end{appendix}

\printbibliography

\end{document}